\documentclass[12pt,twoside]{amsart}
\usepackage{amssymb,amsmath,amsfonts,amsthm}
\usepackage{mathrsfs}
\usepackage[X2,T1]{fontenc}
\usepackage[applemac]{inputenc}
\usepackage{cite}
\usepackage{latexsym}
\usepackage{verbatim}
\usepackage{soul}
\usepackage[dvipsnames]{xcolor}
\usepackage{hyperref}
\def\supp{\mathop{\rm supp}\nolimits}

\def\supp{\mathop{\rm supp}\nolimits}

\def\R{\mathbb R}

\def\N{\mathbb N}

\newcommand\dslash{d\llap {\raisebox{.9ex}{$\scriptstyle-\!$}}}
\usepackage{colortbl}

\newcommand{\beqsn}{\arraycolsep1.5pt\begin{eqnarray*}}
	\newcommand{\eeqsn}{\end{eqnarray*}\arraycolsep5pt}
\newcommand{\beqs}{\arraycolsep1.5pt\begin{eqnarray}}
	\newcommand{\eeqs}{\end{eqnarray}\arraycolsep5pt}

\newtheorem{theorem}{Theorem}
\newtheorem{lemma}{Lemma}

\newtheorem{proposition}{Proposition}
\newtheorem{definition}{Definition}

\newtheorem{remark}{Remark}

\catcode`\@=11

\renewcommand{\section}%
{\setcounter{equation}{0}\@startsection {section}{1}{\z@}{-3.5ex plus -1ex
		minus -.2ex}{2.3ex plus .2ex}{\Large\bf}}

\title[ ]{Some remarks on the Cauchy Problem for Schr\"odinger type equations in Gelfand-Shilov spaces}

\author[A. Arias Junior]{Alexandre Arias Junior}
\address{Dipartimento di Matematica ``G. Peano'' \\Universit\`a di Torino\\
	Via Carlo Alberto 10\\
	10123 Torino\\
	Italy}
\email{alexandre.ariasjunior@unito.it}

\begin{document}
	
	%\def\thefootnote{}
	%\footnote{} 
	
	\begin{abstract}
	We consider the Cauchy problem for Schr\"odinger type operators. Under a suitable decay assumption on the imaginary part of the first order coefficients we prove well-posedness of the Cauchy problem in Gelfand-Shilov classes. We also discuss the optimality of our result through some examples.
	\end{abstract}
	
	\maketitle
	
	\section{Introduction and results}
	
	We consider the Cauchy problem 
	\begin{equation}\label{CP}
		\begin{cases}
			S u(t,x) = f(t,x), \quad t \in [0,T], x \in \R^{n}, \\
			u(0,x) = g(x), \quad \quad \, x \in \R^{n},
		\end{cases}
	\end{equation}
	for the operator 
	$$
	S = D_t - \Delta_x + \sum_{j=0}^{n} a_j(t,x)D_{x_j} + b(t,x), \quad t \in [0,T], x \in \R^{n},
	$$
	where $a_j(t,x) \in C([0,T];\mathcal{B}^{\infty}(\R^{n}))$, $j=1,\ldots,n$, and $\mathcal{B}^{\infty}(\R^{n})$ denotes the space of all complex-valued smooth functions which are bounded in $\R^{n}$ together with all their derivatives. The operator $S$ is the typical example of non-kowalewskian operator which is not parabolic and it is known in the literature as Schr\"odinger type operator. 
	
	It is well known that if all the coefficients $a_j$ are real-valued, then the Cauchy problem \eqref{CP} is well-posed in $L^{2}(\R^{n})$ and in $H^{m}(\R^{n})$ spaces, where $H^{m}(\R^{n})$ stands for the standard $L^2-$based Sobolev spaces. On the other hand, if some coefficient $a_j$ has a non-identically null imaginary part, then $L^{2}(\R^{n})$ or $H^{\infty}(\R^{n}) := \cap_{m \in \R} H^{m} (\R^{n})$ well-posedness may fail. Indeed, when all $a_j(t,x) = a_j(x)$, in \cite{ichinose_remarks_cauchy_problem_schrodinger_necessary_condition} we find that the inequality 
	\begin{equation}\label{eq_necessary_cond}
		\sup_{x \in \R^{n}, \omega \in S^{n-1}} \left| \int_{0}^{\rho} \sum_{j=1}^{n} Im\, a_j(x+\omega\theta) \omega_j d\theta  \right| \leq Mlog(1+\rho) + N, \quad \forall\, \rho \geq 0, 
	\end{equation}
	is a necessary condition for well-posedness in $H^{\infty}(\R^{n})$. Setting $M = 0$ in \eqref{eq_necessary_cond} we obtain a necessary condition for well-posedness in $L^2(\R^{n})$, see \cite{mizohata2014}.
	
	\begin{remark}
		When $n = 1$ the inequality \eqref{eq_necessary_cond} turns out to be also a sufficient condition for $H^{\infty}$ well-posedness, see \cite{Ichinose_sufficient}. Besides, if $M = 0$ and $n = 1$, then \eqref{eq_necessary_cond} is a sufficient condition for $L^2$ well-posedness, see \cite{mizohata2014}.
	\end{remark}
	
	Imposing the following decay condition on the imaginary parts of the coefficients $a_j$
	\begin{equation}\label{eq_decay_assumption}
		|Im\,a_j(t,x)| \leq C \langle x \rangle^{-\sigma}, \, j = 1, \ldots, n, \quad \text{where} \,\, \langle x \rangle^{2} := 1+|x|^{2},
	\end{equation}
	for some $\sigma \in (0,1)$ and $C > 0$, from the aforementioned necessary condition neither $L^2$ nor $H^{\infty}$ well-posedness for the problem \eqref{CP} holds. In fact, this type of decay assumption leads to investigate the problem in suitable Gevrey classes of functions.
	
	For a given $\theta > 1$ we shall denote by $G^{\theta}(\R^{n})$ the set of all smooth functions $f(x)$ such that 
	$$
	\sup_{x \in \R^{n}, \beta \in \N^{n}_{0}} |\partial^{\beta}_{x} f(x)| h^{-|\beta|} \beta!^{-\theta} < \infty,
	$$
	for some $h > 0$. The space $G^{\theta}(\R^{n})$ is commonly referred as the space of uniform Gevrey functions of order $\theta$. Next, for $\rho > 0$ we consider 
	$$
	H^{0}_{\rho;\theta}(\R^{n}) := \{ u \in \mathscr{S}'(\R^{n}) : e^{\rho \langle D \rangle^{\frac{1}{\theta}}} u \in L^{2}(\R^{n}) \},
	$$
	where $\mathscr{S}'(\R^{n})$ stands for the space of tempered distributions and $e^{\rho \langle D \rangle^{\frac{1}{\theta}}}$ denotes the Fourier multiplier with symbol $e^{\rho\langle \xi \rangle^\frac{1}{\theta}}$.
	The space $H^{0}_{\rho;\theta}(\R^{n})$ endowed with the inner product
	$$ 
	\langle u , v \rangle_{H^{0}_{\rho;\theta}} := \langle e^{\rho \langle D \rangle^{\frac{1}{\theta}}} u, e^{\rho \langle D \rangle^{\frac{1}{\theta}}} v \rangle_{L^{2}}, \quad u, v \in H^{0}_{\rho;\theta}(\R^{n}),
	$$
	is a Hilbert space and we denote the induced norm by $\|\cdot\|_{H^{0}_{\rho;\theta}}$. We obviously have $H^{0}_{\rho;\theta}(\R^{n}) \subset H^{\infty}(\R^{n})$ and it is possible to prove that $H^{0}_{\rho;\theta}(\R^{n}) \subset G^{\theta}(\R^{n})$. The functional setting where the Gevrey well-posedness results take place is given by 
	$$
	H^{\infty}_{\theta}(\R^{n}) = \bigcup_{\rho > 0} H^{0}_{\rho;\theta}(\R^{n}).
	$$
	
	\begin{definition}[$H^{\infty}_{\theta}$ well-posedness]
		We say that the Cauchy problem \eqref{CP} is $H^{\infty}_{\theta}$ well-posed when for any given $\rho > 0$ there exist $\tilde{\rho} > 0$ and a constant $C := C(\rho,T) > 0$ such that for all $f \in C([0,T];H^{0}_{\rho;\theta}(\R^{n}))$ and $g \in H^{0}_{\rho;\theta}(\R^{n})$ there exists a unique solution $u \in C^{1}([0,T];H^{0}_{\tilde{\rho};\theta}(\R^{n}))$ and the following energy inequality holds
		\begin{equation}\label{eq_energy_estimate_Gevrey}
			\| u(t,\cdot)\|^{2}_{H^{0}_{\tilde{\rho};\theta}} \leq C \left\{ \|g\|^{2}_{H^{0}_{\rho;\theta}} + \int_{0}^{t} \|f(\tau,\cdot)\|^{2}_{H^{0}_{\rho;\theta}} d\tau \right\}.
		\end{equation}
	\end{definition}
	
	In \cite{KB} (see Theorem $1.1$) the authors proved the following result. 
	
	\begin{theorem}\label{KB_Gevrey_theorem}
		Assume that the coefficients $a_j$ and $b$ satisfy
		$$
		|\partial^{\beta}_{x} a_j(t,x)| + |\partial^{\beta}_{x}b(t,x)| \leq C A^{|\beta|} \beta!^{\theta_{0}}, \quad t \in [0,T], x \in \R^{n}, \beta \in \N_{0}^{n},
		$$
		for some constants $C, A > 0$ and for some $\theta_0 > 1$. Assume moreover that condition \eqref{eq_decay_assumption} is fulfilled for some $\sigma \in (0,1)$. Then, if $\theta \in [\theta_0, \frac{1}{1-\sigma})$ the Cauchy problem \eqref{CP} is $H^{\infty}_{\theta}$ well-posed.
	\end{theorem}
	
	In the limit case $\theta = \frac{1}{1-\sigma}$ local in time $H^{\infty}_{\theta}$ well-posedness for \eqref{CP} holds. We point out that the upper bound $\frac{1}{1-\sigma}$ for the indices $\theta$ in the above theorem is sharp. Indeed, if we consider the model operator in one space dimension 
	\begin{equation}\label{eq_model_opeartor}
			M = D_t + D^{2}_{x} + i\langle x \rangle^{-\sigma}D_{x},		
	\end{equation}
	then Theorem $2$ of \cite{dreher} implies $H^{\infty}_{\theta}$ ill-posedness for the Cauchy problem associated with \eqref{eq_model_opeartor} when $\theta > \frac{1}{1-\sigma}$.
	
	\medskip
	
	In this paper we are interested in the well-posedness of \eqref{CP} in the so called Gelfand-Shilov spaces. Theses spaces were introduced by I.M. Gelfand and G.E. Shilov in \cite{GS2} making part of a larger class called spaces of type $\mathcal{S}$. More into the point, for $\theta > 1$ and $s \geq 1$ we denote by $\mathcal{S}^{\theta}_{s}(\R^{n})$ the space of all smooth functions $f(x)$ such that 
	$$
	\sup_{x\in \R^{n}, \beta \in \N^{n}_{0}} |\partial^{\beta}_{x} f(x)| C^{-|\beta|} \beta!^{-\theta}e^{c|x|^{\frac{1}{s}}} < \infty, 
	$$
	for some $C, c >0$. So, the space $\mathcal{S}^{\theta}_{s}(\R^{n})$ can be viewed as the space of all uniform Gevrey functions of order $\theta$ which, together with all their derivatives, decay as $e^{-c|x|^{\frac{1}{s}}}$ at infinity. 
	
	It is convenient to describe the Gelfand-Shilov classes in terms of a suitable scale of weighted Sobolev spaces. For $\rho = (\rho_1, \rho_2)$, $\rho_1 > 0$, $\rho_2 > 0$ we consider
	$$
	H^{0}_{\rho;\theta,s}(\R^{n}) = \{ u \in \mathscr{S}'(\R^{n}): e^{\rho_2\langle x \rangle^{\frac{1}{s}}} e^{\rho_1 \langle D \rangle^{\frac{1}{\theta}}} u \in L^{2}(\R^{n}) \}.
	$$
	The space $H^{0}_{\rho;\theta,s}(\R^{n})$ is known as Gelfand-Shilov Sobolev space and defines a Hilbert space when endowed with the following inner product
	$$ 
	\langle u , v \rangle_{H^{0}_{\rho;\theta,s}} := \langle e^{\rho_2\langle x \rangle^{\frac{1}{s}}} e^{\rho_1 \langle D \rangle^{\frac{1}{\theta}}} u, e^{\rho_2\langle x \rangle^{\frac{1}{s}}} e^{\rho_1 \langle D \rangle^{\frac{1}{\theta}}} v \rangle_{L^{2}}, \quad u, v \in H^{0}_{\rho;\theta,s}(\R^{n}),
	$$
	and the induced norm is denoted by $\|\cdot\|_{H^{0}_{\rho;\theta,s}}$. Then we express the Gelfand-Shilov classes as
	$$
	\mathcal{S}^{\theta}_{s}(\R^{n}) = \bigcup_{\rho_1, \rho_2 > 0} H^{0}_{\rho;\theta,s}(\R^{n}).
	$$
	
	Now we precisely define what we mean by $\mathcal{S}^{\theta}_{s}$ well-posedness. 
	
	\begin{definition}[$\mathcal{S}^{\theta}_{s}$ well-posedness]
		We say that the Cauchy problem \eqref{CP} is $\mathcal{S}^{\theta}_{s}$ well-posed when for any given $\rho = (\rho_1, \rho_2), \rho_1, \rho_2 > 0$ there exist $\tilde{\rho} = (\tilde{\rho}_1, \tilde{\rho}_2), \tilde{\rho}_1, \tilde{\rho}_2 > 0$ and a constant $C := C(\rho,T) > 0$ such that for all $f \in C([0,T];H^{0}_{\rho;\theta,s}(\R^{n}))$ and $g \in H^{0}_{\rho;\theta,s}(\R^{n})$ there exists a unique solution $u \in C^{1}([0,T];H^{0}_{\tilde{\rho};\theta,s}(\R^{n}))$ and the following energy inequality holds 
		\begin{equation}\label{eq_energy_estimate_Gelfand_shilov}
			\| u(t,\cdot)\|^{2}_{H^{0}_{\tilde{\rho};\theta,s}} \leq C \left\{ \|g\|^{2}_{H^{0}_{\rho;\theta,s}} + \int_{0}^{t} \|f(\tau,\cdot)\|^{2}_{H^{0}_{\rho;\theta,s}} d\tau \right\}.
		\end{equation}
	\end{definition}
	
	The main result of this manuscript reads as follows.
	
	\begin{theorem}\label{main_theorem}
		Assume that the coefficients $a_j$ and $b$ satisfy
		$$
		|\partial^{\beta}_{x} a_j(t,x)| + |\partial^{\beta}_{x}b(t,x)| \leq C A^{|\beta|} \beta!^{\theta_{0}}, \quad t \in [0,T], x \in \R^{n}, \beta \in \N_{0}^{n},
		$$
		for some constants $C, A > 0$ and for some $\theta_0 > 1$. Assume moreover that condition \eqref{eq_decay_assumption} is fulfilled for some $\sigma \in (0,1)$. Then, if $\theta \in [\theta_0, \min\{\frac{1}{1-\sigma}, s\})$ the Cauchy problem \eqref{CP} is $\mathcal{S}^{\theta}_{s}$ well-posed.
	\end{theorem}
	
	\begin{remark}\label{remark_limit_case}
		In the limit case $\theta = \min\{\frac{1}{1-\sigma}, s\}$ we obtain local in time $\mathcal{S}^{\theta}_{s}$ well-posedness.
	\end{remark}
	
	Now we discuss the constraints on the parameters $s, \theta$ and $\sigma$ in Theorem \ref{main_theorem} by considering some particular cases. For the model operator \eqref{eq_model_opeartor} we state the second main result of this paper.
	
	\begin{proposition}\label{prop_model_problem}
		If the Cauchy problem for the operator $M$ given by \eqref{eq_model_opeartor} is $\mathcal{S}^{\theta}_{s}$ well-posed then $\max\{\frac{1}{\theta}, \frac{1}{s}\} \geq 1-\sigma$. 
	\end{proposition}
	
	From Proposition \ref{prop_model_problem} it follows that when $s \geq \theta$ we cannot allow $\theta > \frac{1}{1-\sigma}$ in Theorem \ref{main_theorem}.
	
	Now we investigate the case $s < \theta$. For this latter case we consider the Cauchy problem for the Schr\"odinger operator in one space dimension
	\begin{equation}\label{CP_schrodinger_operator}
	\begin{cases}
		\{D_t+D^2_x\}u(t,x) = 0, \quad t \in [0,T], x \in \R, \\
		u(0,x) = g(x), \qquad \qquad x \in \R,
	\end{cases}
	\end{equation}
	where $g \in \mathcal{S}^{\theta}_{s}(\R)$. The solution of \eqref{CP_schrodinger_operator} can be written explicitly as
	\begin{equation}\label{eq_sol_CP_schrodinger_operator}
	u_{g}(t,x) = \int_{\R} e^{i\xi x} e^{-i\xi^2t} \widehat{g}(\xi) \dslash\xi,
	\end{equation}
	where $\dslash \xi = (2\pi)^{-1} d\xi$ and $\widehat{g}$ denotes the Fourier transform of $g$.
	The Gevrey regularity of $u_g$ is given by the index $\theta$, while integration by parts and costumary computations give that $u_{g}$ decays as $e^{-c|x|^{\frac{1}{\max\{s,\theta\}}}}$ at $|x| \to \infty$. So, we conclude that $u_g(t, \cdot) \in \mathcal{S}^{\theta}_{\max\{s,\theta\}} (\R)$. Thus, we obtain that the Cauchy problem \eqref{CP_schrodinger_operator} is $\mathcal{S}^{\theta}_{s}$ well-posed when $s \geq \theta$. 
	
	Notice that the Fourier transform in $x$ of $u_g(t,x)$ is given by $\widehat{u}_g(t,\xi) = e^{-i\xi^2t} \widehat{g}(\xi)$ and from Theorem $6.1.2$ of \cite{nicola_rodino_global_pseudo_diffferential_calculus_on_euclidean_spaces} we have (for any fixed $t > 0$)
	$$
	g \in \mathcal{S}^{\theta}_{s}(\R) \iff \widehat{g}\in\mathcal{S}^{s}_{\theta}(\R) \quad \text{and} \quad u_g (t, x) \in \mathcal{S}^{\theta}_{s}(\R) \iff e^{-i\xi^2t} \widehat{g}(\xi) \in \mathcal{S}^{s}_{\theta}(\R).
	$$
	Hence, if the Cauchy problem \eqref{CP_schrodinger_operator} is well-posed in $\mathcal{S}^{\theta}_{s}(\R)$ then for any fixed $t > 0$ the function $h(\xi) = e^{-it\xi^2}$ defines a multiplier in $\mathcal{S}^{s}_{\theta}(\R)$. The last result of this paper, Proposition \ref{main_prop_2} below, implies $\mathcal{S}^{\theta}_{s}$ ill-posedness of \eqref{CP_schrodinger_operator} provided that $1 \leq s < \theta$. As a consequence, we conclude that $s < \theta$ is not allowed in Theorem \ref{main_theorem}.
	
	\begin{proposition}\label{main_prop_2}
		Let $1 \leq s < \theta$. For any fixed $t > 0$ the function $h(\xi) = e^{-it\xi^2}$ does not define a multiplier in $\mathcal{S}^{s}_{\theta}(\R)$.
	\end{proposition}
	
	\begin{remark}
		Another consequence of Proposition \ref{main_prop_2} is that if $g \in \mathcal{S}^{\theta}_{s}(\R)$ with $1\leq s < \theta$, then the solution of \eqref{CP_schrodinger_operator} may present a weaker decay at infinity with respect to the initial datum.
	\end{remark}
	
	\begin{remark}
		In \cite{ascanelli_cappiello_schrodinger_equations_Gelfand_shillov}, under slightly different assumptions on the coefficients $a_j$ and $b$, the authors studied the problem \eqref{CP} in Gelfand-Shilov type spaces. They proved that if the initial data belong to a suitable Gelfand-Shilov class, then there exists a solution with a possible exponential growth at $|x| \to \infty$. There, no well-posedness results in $\mathcal{S}^{\theta}_{s}$ were achieved.
	\end{remark}
	
	%The paper is organized as follows: in Section \ref{section_main_theorem} we prove Theorem \ref{main_theorem}, in Section \ref{section_model_case} we prove Proposition \ref{prop_model_problem} and in Section \ref{section_proof_prop_2} we prove Proposition \ref{main_prop_2}.
	
	\section{Proof of Theorem \ref{main_theorem}}\label{section_main_theorem}
		
	We begin performing the conjugation of the operator $S$ by the exponential $e^{\delta\langle x \rangle^{\frac{1}{s}}}$, where $\delta > 0$. For any $j = 1, \ldots, n$, we have
		\begin{equation*}
			e^{\delta\langle x \rangle^{\frac{1}{s}}}\, D_{x_j}\, e^{-\delta\langle x \rangle^{\frac{1}{s}}} = D_{x_j} - \delta D_{x_j}\langle x \rangle^{\frac{1}{s}},
		\end{equation*}
		\begin{equation*}
			e^{\delta \langle x \rangle^{\frac{1}{s}}}\, \partial^{2}_{x_j}\, e^{-\delta\langle x \rangle^{\frac{1}{s}}} =
			\partial^{2}_{x_j} - 2\delta \partial_{x_j} \langle x \rangle^{\frac{1}{s}} \partial_{x_j} - \delta \partial^{2}_{x_j}\langle x \rangle^{\frac{1}{s}} + (\delta\partial_{x_j} \langle x \rangle^{\frac{1}{s}})^{2}.
		\end{equation*}
	Hence
		\begin{equation}\label{eq_conjugation_second_order}
			e^{\delta\langle x \rangle^{\frac{1}{s}}} \Delta_x e^{-\delta\langle x \rangle^{\frac{1}{s}}} = \Delta_x -i2\delta\sum_{j=1}^{n} \partial_{x_j}\langle x \rangle^{\frac{1}{s}} D_{x_j} - \delta\Delta_x \langle x \rangle^{\frac{1}{s}} + \delta^{2}\sum_{j=1}^{n} (\partial_{x_j}\langle x \rangle^{\frac{1}{s}})^{2},
		\end{equation}
		\begin{equation}\label{eq_conjugation_first_order}
			e^{\delta\langle x \rangle^{\frac{1}{s}}} \left\{ \sum_{j=1}^{n} a_j(t,x) D_{x_j}\right\} e^{-\delta \langle x \rangle^{\frac{1}{s}}} =
			\sum_{j=1}^{n} a_j(t,x) D_{x_j} - \delta \sum_{j=1}^{n} a_j(t,x)D_{x_j}\langle x \rangle^{\frac{1}{s}}.
		\end{equation}
	Therefore, from \eqref{eq_conjugation_second_order} and \eqref{eq_conjugation_first_order} we get
	\begin{equation}
		S_\delta := e^{\delta \langle x \rangle^{\frac{1}{s}}} S e^{-\delta\langle x \rangle^{\frac{1}{s}}} = 
		D_t + \Delta_x + \sum_{j=1}^{n} a_{j,\delta}(t,x) D_{x_j} + b_{\delta}(t,x),
	\end{equation}
	where
	\begin{equation}
		a_{j,\delta}(t,x) := a_j(t,x) - i2\delta \partial_{x_j} \langle x \rangle^{\frac{1}{s}}, \quad  j = 1, \ldots, n,
	\end{equation}
	\begin{equation}
		b_{\delta}(t,x) := b(t,x) - \delta\Delta_x \langle x \rangle^{\frac{1}{s}} + \delta^{2}\sum_{j=1}^{n} (\partial_{x_j}\langle x \rangle^{\frac{1}{s}})^{2}
		- \delta \sum_{j=1}^{n} a_j(t,x)D_{x_j}\langle x \rangle^{\frac{1}{s}}.
	\end{equation}

	Note that the operator $S_{\delta}$ possesses Gevrey regular coefficients of order $\theta_0$.  Besides, the imaginary part of the coefficients of order one satisfy
	\begin{equation}
		|Im \, a_{j,\delta}(t,x)| \leq |Im \, a_j(t,x)| + 2\delta|\partial_{x_j}\langle x \rangle^{\frac{1}{s}}| \leq C \langle x \rangle^{-\min\{\sigma, 1-\frac{1}{s}\}},
	\end{equation}
	for some positive constant $C > 0$ depending on $\delta, s$ and the imaginary part of the coefficients $a_j$, $j=1,\ldots, n$. So, Theorem \ref{KB_Gevrey_theorem} implies that the auxiliary Cauchy problem 
	\begin{equation}\label{CP_conjugated_problem}
		\begin{cases}
			S_{\delta} u(t,x) = e^{\delta \langle x \rangle^{\frac{1}{s}}}f(t,x), \quad t \in [0,T], x \in \R^{n}, \\
			u(0,x) = e^{\delta\langle x \rangle^{\frac{1}{s}}}g(x), \quad \quad \,\,\, x \in \R^{n},
		\end{cases}
	\end{equation}
	is $H^{\infty}_{\theta}$ well-posed, provided that $\theta \in [\theta_0, \min\{1-\sigma, \frac{1}{s}\})$.
	
	Given those preliminaries we are finally ready to prove Theorem \ref{main_theorem}.
	
	\begin{proof}
		Let $\rho = (\rho_1, \rho_2)$, with $\rho_1, \rho_2 > 0$, and assume that the initial data for $\eqref{CP}$ satisfy 
		$$
		g \in H^{0}_{\rho;\theta,s}(\R^{n})\quad  \text{and} \quad f \in C([0,T]; H^{0}_{\rho;\theta,s}(\R^{n})).
		$$ 
		Then, for any $\delta \in (0, \rho_2)$ we obtain 
		$$
		g_\delta := e^{\delta\langle x \rangle^{\frac{1}{s}}} g \in H^{0}_{(\rho_1,\rho_2-\delta);\theta,s}(\R^{n})\quad \text{and} \quad 
		f_{\delta} := e^{\delta\langle x \rangle^{\frac{1}{s}}} f \in C([0,T]; H^{0}_{(\rho_1, \rho_2-\delta);\theta,s}(\R^{n})).
		$$
		
		Since $\rho_2-\delta > 0$ we have $g_\delta \in H^{0}_{\rho_1;\theta}(\R^{n})$ and $f_{\delta} \in C([0,T];H^{0}_{\rho_1;\theta}(\R^{n}))$. Under the assumption $\theta \in [\theta_0, \min\{\frac{1}{1-\sigma},s\}]$ the auxiliary Cauchy problem \eqref{CP_conjugated_problem} is $H^{\infty}_{\theta}$ well-posed. So, there exists a unique solution $v \in C^1([0,T]; H^{0}_{\tilde{\rho}_{1};\theta}(\R^{n}))$ of \eqref{CP_conjugated_problem} with initial data $g_\delta$ and $f_\delta$ and the solution $v$ satisfies
		\begin{equation}\label{eq_energy_solution_auxilliary_problem}
			\| v(t,\cdot)\|^{2}_{H^{0}_{\tilde{\rho}_1;\theta}} \leq C(\rho_1,T) \left\{ \|g_{\delta}\|^{2}_{H^{0}_{\rho_1;\theta}} + \int_{0}^{t} \|f_{\delta}(\tau,\cdot)\|^{2}_{H^{0}_{\rho_1;\theta}} d\tau \right\}.
		\end{equation}
	
		Next we define $u(t,x) = e^{-\delta\langle x \rangle^{\frac{1}{s}}} v(t,x)$. Then $u \in C^{1}([0,T]; H^{0}_{(\tilde{\rho}_1, \delta);\theta,s}(\R^{n}))$ and $u$ solves the Cauchy problem \eqref{CP}. Moreover, in view of \eqref{eq_energy_solution_auxilliary_problem} we get 
		\begin{align*}
			\|u(t,\cdot)\|_{H^{0}_{(\tilde{\rho}_1,\delta);\theta,s}} &\leq C(\delta) \|v(t,\cdot)\|_{H^{0}_{\tilde{\rho}_1;\theta}} \\
			&\leq C(\rho_1,\delta,T) \left\{ \|g_{\delta}\|^{2}_{H^{0}_{\rho_1;\theta}} + \int_{0}^{t} \|f_{\delta}(\tau,\cdot)\|^{2}_{H^{0}_{\rho_1;\theta}} d\tau \right\} \\
			&\leq C(\rho_1,\delta,T) \left\{ \|g\|^{2}_{H^{0}_{(\rho_1, \delta);\theta,s}} + \int_{0}^{t} \|f(\tau,\cdot)\|^{2}_{H^{0}_{(\rho_1,\delta);\theta,s}} d\tau \right\} \\
			&\leq C(\rho_1,\delta,T) \left\{ \|g\|^{2}_{H^{0}_{\rho;\theta,s}} + \int_{0}^{t} \|f(\tau,\cdot)\|^{2}_{H^{0}_{\rho;\theta,s}} d\tau \right\},
		\end{align*}
		for some positive constant $C(\rho_1, \delta, T)$ depending on $\rho_1$, $\delta < \rho_2$ and $T$.
		
		In order to conclude the uniqueness of the solution, let $u_j \in C^1([0,T];H^{0}_{(\tilde{\rho}_1, \delta);\theta, s}(\R^{n}))$, $j = 1, 2$, $\delta < \rho_2$, two solutions for the problem \eqref{CP}. Then, for any $\tilde{\delta} \in (0, \delta)$, $e^{\tilde{\delta}\langle x \rangle^{\frac{1}{s}}}u_j$, $j=1,2$, are solutions of the Cauchy problem \eqref{CP_conjugated_problem}, with $\delta$ replaced by $\tilde{\delta}$, and belong to $C^1([0,T];H^{0}_{\tilde{\rho}_1;\theta}(\R^{n}))$. From the $H^{\infty}_{\theta}$ well-posedness of \eqref{CP_conjugated_problem} we conclude $e^{\tilde{\delta}\langle x \rangle^{\frac{1}{s}}}u_1 = e^{\tilde{\delta}\langle x \rangle^{\frac{1}{s}}}u_2$, that is, $u_1 = u_2$.
	\end{proof}

	\section{Proof of Proposition \ref{prop_model_problem}}\label{section_model_case}
	
	In this section we deal with the Cauchy problem for the model operator $M$ given by \eqref{eq_model_opeartor}, that is, for $t \in [0,T]$ and $x \in \R$,
	\begin{equation}\label{cauchy_problem_Model}
		\begin{cases}
			\{D_t + D^2_x + i\langle x \rangle^{-\sigma}D_x\}u(t,x) = 0, \\
			u(0,x)= g(x),
		\end{cases}
	\end{equation}
	where $\sigma > 0$. Following an argument inspired by \cite{AAC_necessary} and \cite{dreher} we shall prove that if the Cauchy problem \eqref{cauchy_problem_Model} is $\mathcal{S}^{\theta}_{s}$ well-posed, then $(1-\sigma) > \max \{\frac{1}{\theta}, \frac{1}{s}\}$ leads to a contradiction. 
	
	Before proceeding with the proof we need to establish some notations and state some results. As usual, for any $m \in \R$, $S^{m}_{0,0}(\R)$ stands for the space of all functions $p \in C^{\infty}(\R^2)$ such that for any $\alpha, \beta \in \N_0$ there exists $C_{\alpha,\beta} > 0$ such that 
	$$
	|\partial^{\alpha}_{\xi} \partial^{\beta}_{x} p(x,\xi)| \leq C_{\alpha,\beta} \langle \xi \rangle^{m}.
	$$
	The topology in $S^{m}_{0,0}(\R)$ is induced by the following family of seminorms
	$$
	|p|^{(m)}_\ell := \max_{\alpha \leq \ell, \beta \leq \ell} \sup_{x, \xi \in \R} |\partial^{\alpha}_{\xi} \partial^{\beta}_{x} p(x,\xi)| \langle \xi \rangle^{-m}, \quad p \in S^{m}_{0,0}(\R), \, \ell \in \N_0.
	$$
	We associate to every symbol $p \in S^{m}_{0,0}(\R)$ the continuous operator $p(x,D): \mathscr{S}(\R) \to \mathscr{S}(\R)$ (Schwartz space of rapidly decreasing functions), known as pseudodifferential operator, defined by 
	$$
	p(x,D) u(x) = \int e^{i\xi x} p(x,\xi) \widehat{u}(\xi) \dslash\xi, \quad u \in \mathscr{S}(\R).
	$$
	
	The next result gives the action of operators coming from symbols $S^{m}_{0,0}(\R)$ in the standard sobolev spaces $H^{s}(\R)$, $s \in \R$. For a proof we address the reader to Theorem $1.6$ on page $224$ of \cite{Kumano-Go}.
	
	\begin{theorem}\label{theorem_Calderon_Vaillancourt}[Calder\'on-Vaillancourt]
		Let $p \in S^{m}_{0,0}(\R)$. Then for any real number $s \in \R$ there exist $\ell := \ell(s,m) \in \N_0$ and $C:= C_{s,m} > 0$ such that 
		\begin{equation*}
			\| p(x,D)u \|_{H^{s}(\R)} \leq C |p|^{(m)}_{\ell} \| u \|_{H^{s+m}(\R)}, \quad \forall \, u \in H^{s+m}(\R).
		\end{equation*}
		Besides, when $m = s = 0$ we can replace $|p|^{(m)}_{\ell}$ by
		\begin{equation*}
			\max_{\alpha, \beta \leq 2} \sup_{x, \xi \in \R} |\partial^{\alpha}_{\xi} \partial^{\beta}_{x} p(x,\xi)|.
		\end{equation*}
	\end{theorem}
	
	Now we consider the algebra porperties of $S^{m}_{0,0}(\R)$ with respect the composition of operators. Let $p_j \in S^{m_j}_{0,0}(\R)$, $j = 1, 2$, and define 
	\begin{align}\label{eq_symbol_of_composition}
		q(x,\xi) &= Os- \iint e^{-iy \eta} p_1(x,\xi+\eta)p_2(x+y,\xi) dy \dslash\eta \\
		&= \lim_{\varepsilon \to 0} \iint e^{-iy\eta} p_1(x,\xi+\eta)p_2(x+y,\xi) e^{-\varepsilon^2y^2} e^{-\varepsilon^2\eta^2} dy \dslash\eta. \nonumber
	\end{align} 
	We often write $q(x,\xi) = p_1(x,\xi) \circ p_2(x,\xi)$. Then we have the following theorem (for a proof see Lemma $2.4$ on page $69$ and Theorem $1.4$ on page $223$ of \cite{Kumano-Go}).
	
	\begin{theorem}
		Let $p_j \in S^{m_j}_{0,0}(\R)$, $j = 1, 2$, and consider $q$ defined by \eqref{eq_symbol_of_composition}. Then $q \in S^{m_1+m_2}_{0,0}(\R)$ and $q(x,D) = p_1(x,D)p_2(x,D)$. Moreover, the symbol $q$ has the following asymptotic expansion
		\begin{align}\label{eq_asymptotic_expansion_formula}
			q(x,\xi) = \sum_{\alpha < N} \frac{1}{\alpha!} \partial^{\alpha}_{\xi}p_{1}(x,\xi)D^{\alpha}_{x}p_2(x,\xi) + r_N(x,\xi), 
		\end{align} 
		where 
		$$
		r_N(x,\xi) = N\int_{0}^{1} \frac{(1-\theta)^{N-1}}{N!} \, Os - \iint e^{-iy\eta} \partial^{N}_{\xi} p_1(x,\xi+\theta\eta) D^{N}_{x} p_2(x+y,\xi)  dy\dslash\eta \, d\theta,
		$$
		and the seminorms of $r_N$ may be estimated in the following way: for any $\ell_{0} \in \N_0$ there exists $\ell_{1} := \ell_1(\ell_{0}) \in \N_0$ such that 
		$$
		|r_N|^{(m_1+m_2)}_{\ell_0} \leq C_{\ell_{0}} |\partial^{N}_{\xi}p_1|^{(m_1)}_{\ell_{1}} |\partial^{N}_{x}p_2|^{(m_2)}_{\ell_{1}}.
		$$
	\end{theorem} 
	
	The last theorem that we recall is the so-called sharp G{\aa}rding inequality. To state this result we need to define the standard H\"ormander classes of symbols.  We say that $p \in S^{m}(\R^{2})$ if $p \in C^{\infty}(\R^2)$ and for any $\alpha, \beta \in \N_0$ there exists $C_{\alpha,\beta} > 0$ such that 
	\begin{equation}\label{eq_def_symbol}
	|\partial^{\alpha}_{\xi} \partial^{\beta}_{x} p(x,\xi)| \leq C_{\alpha,\beta} \langle \xi \rangle^{m-\alpha}.
	\end{equation}
	
	\begin{theorem}\label{theorem_SG_sharp_garding}[sharp G{\aa}rding inequality]
		Let $p \in S^{m}(\R^{2})$. If $Re\, p(x,\xi) \geq 0$ for all $x,\xi \in \R$ then there exists a constant $C > 0$, depending on a finite number of the constants $C_{\alpha,\beta}$ in \eqref{eq_def_symbol}, such that 
		$$
		Re\, \langle p(x,D) u, u \rangle_{L^{2}} \geq -C \| u \|_{H^{\frac{m-1}{2}}}, \quad u \in \mathscr{S}(\R).
		$$
	\end{theorem} 
	
	Now we return to the proof of Proposition \ref{prop_model_problem}. As we mentioned before, the idea is to argue by contradiction. We start then defining the main ingredients to get the desired contradiction.
	
	Let $\phi \in G^{\theta}(\R)$ determined by 
	\begin{equation}\label{eq_fourier_transform_phi}
	\widehat{\phi}(\xi) = e^{-\rho_0 \langle \xi \rangle^{\frac{1}{\theta}}},
	\end{equation}
	for some $\rho_0 > 0$. We claim that such $\phi$ belongs to $\mathcal{S}^{\theta}_{1}(\R)$. Indeed, in view of Proposition $6.1.7$ of \cite{nicola_rodino_global_pseudo_diffferential_calculus_on_euclidean_spaces} we only need to prove that
	$$
	\phi(x) = \int e^{i\xi x} \widehat{\phi}(\xi) \dslash\xi
	$$
	satisfies $|\phi(x)| \leq C e^{-c|x|}$ for some positive constants $C, c > 0$. To verify that, we integrate by parts to get
	$$
	x^{\beta} \phi(x) = (-1)^{\beta} \int e^{i\xi x} D^{\beta}_{\xi} e^{-\rho_0 \langle \xi\rangle^{\frac{1}{\theta}}} \dslash\xi.
	$$
	Then, Fa\`{a} di Bruno formula and costumary factorial inequalities imply 
	\begin{align*}
		|x^{\beta}\phi(x)| &\leq \int \sum_{j=1}^{\beta} \frac{e^{-\rho_0 \langle \xi \rangle^{\frac{1}{\theta}}}}{j!} \sum_{\overset{\beta_1 +\cdots + \beta_j= \beta}{\beta_\ell \geq 1} } \frac{\beta!}{\beta_1!\ldots \beta_j!} \prod_{\ell=1}^{j} |\rho_0\partial^{\beta_{\ell}}_{\xi} \langle \xi \rangle^{\frac{1}{\theta}}| \dslash\xi \\
		&\leq C^{\beta}_{\rho_0} \beta!, \quad \forall\, x \in \R, \, \beta \in \N_0, 
	\end{align*}
	which allows to conclude our claim. Hence, there exists $\rho = (\rho_1,\rho_2)$, $\rho_1,\rho_2 > 0$ such that $\phi \in H^{0}_{\rho;\theta,s}(\R)$ for all $s \geq 1$.

	Next we take an arbitrary sequence $(\sigma_k)$ of positive real numbers such that $\sigma_k \to \infty$, as $k \to \infty$, and then we define
	$$
	\phi_k(x) = e^{-\rho_2 4^{\frac{1}{s}} \sigma^{\frac{1}{s}}_{k}} \phi(x - 4\sigma_k),\quad k \in \N_0.
	$$ 
	So, $\phi_k \in H^{0}_{\rho;\theta,s}(\R)$ for all $k$ and we have the following estimate:
	\begin{align*}
		\|\phi_k\|^{2}_{H^{0}_{\rho;\theta,s}} &= \int e^{2\rho_2 \langle x \rangle^{\frac{1}{s}}} |e^{\rho_1 \langle D \rangle^{\frac{1}{\theta}}} \phi_k(x)|^{2} dx \\
		&=  \int e^{2\rho_2 \langle x + 4\sigma_{k} \rangle^{\frac{1}{s}}}  e^{-2\rho_2 4^{\frac{1}{s}} \sigma^{\frac{1}{s}}_{k}} |e^{\rho_1 \langle D \rangle^{\frac{1}{\theta}}} \phi(x)|^{2} dx \\
		&\leq \int e^{2\rho_2 \langle x \rangle^{\frac{1}{s}}} |e^{\rho_1 \langle D \rangle^{\frac{1}{\theta}}} \phi(x)|^{2} dx \\
		&= \|\phi\|^{2}_{H^{0}_{\rho;\theta, s}} = \, \text{constant}. \\
	\end{align*} 
	That is, the sequence $(\|\phi_k\|_{H^{0}_{\rho;\theta,s}})$ is uniformly bounded in $k$. 
	
	If the problem \eqref{cauchy_problem_Model} is $\mathcal{S}^{\theta}_{s}$ well-posed, for every $k \in \N_0$ let $u_k \in C^{1}([0,T];H^{0}_{\tilde{\rho};\theta,s}(\R))$, $\tilde{\rho} = (\tilde{\rho}_1,\tilde{\rho}_{2})$ with $\tilde{\rho}_1,\tilde{\rho}_2 > 0$, be the solution of \eqref{cauchy_problem_Model} with initial datum $\phi_k$. From the energy inequality we get
	\begin{equation}\label{fat_girl}
	\|u_{k}(t,\cdot)\|_{L^{2}} \leq \|u_{k}(t,\cdot)\|_{H^{0}_{\tilde{\rho};\theta,s}} \leq C_{T,\rho} \|\phi_k\|_{H^{0}_{\rho;\theta,s}} \leq C_{T,\rho} \|\phi\|_{H^{0}_{\rho;\theta,s}}.
	\end{equation}
	Hence, we conclude that the sequence $(\|u_k(t)\|_{L^2})$ is uniformly bounded with respect to both $k \in \N_0$ and $t \in [0,T]$. 
	
	 For $\theta_h > 1$ close to $1$ we consider a Gevrey cutoff function $h \in G^{\theta_h}_{0}(\R)$ such that 
	 $$
	 h(x) = \begin{cases}
	 	1, \,\, |x| \leq \frac{1}{2}, \\
	 	0, \,\, |x| \geq 1,
	 \end{cases}
	 $$
	 $\widehat{h}(0) > 0$  and $\widehat{h}(\xi) \geq 0$ for all $\xi \in \R$. We also consider the sequence of symbols 
	 \begin{equation}\label{eq_def_localization}
	 	w_{k}(x,\xi) = h\left( \frac{ x-4\sigma_{k} }{ \sigma_{k} } \right) h\left(  \frac{\xi - \sigma_k}{ \frac{1}{4}\sigma_{k} } \right).
	 \end{equation}
	 Finally, for $\lambda \in (0,1)$ to be chosen later, $\theta_1 > \theta_h$ (still close to $1$) and 
	\begin{equation}\label{eq_definition_of_N_k}
		N_k := \lfloor \sigma^{\frac{\lambda}{\theta_1}}_{k} \rfloor = \max \{a \in \N_0 : a \leq \sigma^{\frac{\lambda}{\theta_1}}_{k}\}
	\end{equation}
	we introduce the energy 
	\begin{equation}\label{eq_def_energy}
		E_{k}(t) = \sum_{\alpha \leq N_k, \beta \leq N_k} \frac{1}{(\alpha!\beta!)^{\theta_1}} \| w^{(\alpha\beta)}_{k} (x,D) u_k(t,x) \|_{L^{2}} = \sum_{\alpha \leq N_k, \beta \leq N_k} E_{k,\alpha,\beta}(t),
	\end{equation}
	where 
	\begin{equation*}
		w^{(\alpha\beta)}_{k}(x,\xi) = h^{(\alpha)}\left( \frac{x-4\sigma_{k}}{ \sigma_{k} } \right)
		h^{(\beta)} \left( \frac{\xi - \sigma_k}{ \frac{1}{4}\sigma_{k} } \right).
	\end{equation*}
	
	\begin{remark}\label{remark_about_the_w_k}
		We have $w_k(x,\xi) \in S^{0}_{0,0}(\R)$ and the following estimate
		\begin{equation}\label{eq_estimate_seminomrs_w_k}
		|\partial^{\nu}_{x} \partial^{\gamma}_{\xi} w^{(\alpha\beta)}_{k}(x,\xi)|^{(0)}_{\ell} \leq C^{\alpha+\beta+\gamma+\nu+\ell+1} (\alpha!\beta!\gamma!\nu!\ell!^2)^{\theta_h} \sigma^{-\gamma}_{k} \sigma^{-\nu}_{k},
		\end{equation}
		for some constant $C>0$ independent from $k, \alpha, \beta, \nu$ and $\gamma$. We also remark that on the support of $w_k$ it holds
		$$
		\frac{3\sigma_k}{4} \leq \xi \leq \frac{5\sigma_k}{4} \quad \text{and} \quad 3\sigma_k \leq x \leq 5\sigma_k.
		$$
	 	Thus, on the support of $w_{k}$, $\xi$ is comparable with $\sigma_k$ and $x$ is comparable with $\sigma_{k}$. 
	\end{remark}
	
	The following two propositions will play a key role in our proof. 
	
	\begin{proposition}\label{proposition_upper_bound_energy_E_k}
		Let the Cauchy problem \eqref{cauchy_problem_Model} be $\mathcal{S}^{\theta}_{s}$ well-posed. We then have for all $t \in [0,T]$ and $k \in \N_0$ 
		\begin{equation}\label{eq_estimate_from_above_E_k_alpha_beta}
			E_{k,\alpha,\beta} (t) \leq C^{\alpha+\beta+1} \{\alpha!\beta!\}^{\theta_h-\theta_1},
		\end{equation}
		\begin{equation}\label{eq_estimate_from_above_E_k}
			E_k(t) \leq C,
		\end{equation}
		for some positive constant $C > 0$ indepent from $k$ and $t$.
	\end{proposition}
	\begin{proof}
		Calder\'on-Vaillancourt Theorem, \eqref{fat_girl} and \eqref{eq_estimate_seminomrs_w_k} imply 
		\begin{align*}
			\|w^{(\alpha\beta)}_{k}(x,D)u_k\|_{L^{2}} &\leq  C \|u_k\|_{L^{2}} \max_{\alpha, \beta \leq 2} \sup_{x, \xi \in \R} |\partial^{\alpha}_{\xi} \partial^{\beta}_{x} w_k(x,\xi)|  \\
			&\leq C^{\alpha+\beta+1} \{\alpha!\beta!\}^{\theta_h}.
		\end{align*}
		Hence we immediatily see that \eqref{eq_estimate_from_above_E_k_alpha_beta} holds from the definition of $E_{k,\alpha,\beta}$ (cf. \eqref{eq_def_energy}). To obtain \eqref{eq_estimate_from_above_E_k} we recall that $\theta_1 > \theta_h$ and therefore we obtain 
		\begin{align*}
			E_{k}(t) &= \sum_{\alpha \leq N_k, \beta \leq N_k} E_{k,\alpha,\beta}(t) \\
			&< C \underbrace{\sum_{\alpha,\beta < \infty} C^{\alpha+\beta} \{\alpha!\beta!\}^{\theta_h-\theta_1}}_{=\text{constant}}.
		\end{align*}
	\end{proof}

	\begin{proposition}\label{propostion_lower_bound_derivative_energy}
		Let the Cauchy problem \eqref{cauchy_problem_Model} be $\mathcal{S}^{\theta}_{s}$ well-posed. We then have for all $t \in [0,T]$ and all $k$ sufficiently large
		\begin{align*}
			\partial_t E_{k}(t) \geq \Bigg[ c_1 \sigma^{1-\sigma}_{k} - C \sigma^{\lambda}_{k}\Bigg] E_{k}(t) - C^{N_k+1}\sigma^{C-cN_k}_{k},
		\end{align*}
		for some  positive constants $C, c, c_1 >0$ independent from $k$ and $t$.
	\end{proposition}
	
	The proof of Proposition \ref{propostion_lower_bound_derivative_energy} is quite technical and long, so, in order to address the reader as quick as possible to the proof of Proposition \ref{prop_model_problem}, we postpone it to the Subsection \ref{Subsection}.
	
	Given all those preliminaries we are finally ready to prove Proposition \ref{prop_model_problem}.
	
	\begin{proof} 
	Let us denote by 
	$$
	A_k :=  c_1 \sigma^{1-\sigma}_{k} - C \sigma^{\lambda}_{k}, 
	\quad R_k = C^{N+1}\sigma^{C-cN_k}_{k}.
	$$
	If the parameter $\lambda < 1-\sigma$, since $\sigma_k \to \infty$, we obtain for all $k$ sufficiently large
	$$
	A_{k} \geq \frac{c_1}{2} \sigma^{1-\sigma}_{k}.
	$$
	In the next we always shall consider $k$ sufficiently large. Applying Gronwall's inequality to Proposition \ref{propostion_lower_bound_derivative_energy} we obtain
	\begin{align*}
		E_k(t) \geq e^{A_k t} \Bigg\{ E_k(0) - R_k \int^{t}_{0} e^{- A_k \tau} d\tau  \Bigg\}.
	\end{align*}
	Hence, for any fixed $T^* \in (0,T]$,
	\begin{align}\label{eq_almost_there}
		E_k(T^*) \geq e^{T^* \frac{c_1}{2} \sigma^{(p-1)(1-\sigma)}_{k}} \Bigg\{ E_k(0) - T^* R_k  \Bigg\}.
	\end{align}
	
	Now we estimate the terms $R_k$ and $E_k(0)$. Recalling that $N_k = \lfloor \sigma_{k}^{\frac{\lambda}{\theta_1}} \rfloor$ we may estimate $R_k$ from above in the following way
	\begin{equation}\label{eq_almost_there_2}
		R_k \leq C e^{-c \sigma^{\frac{\lambda}{\theta_1}}_{k}}.
	\end{equation}
	To estimate $E_k(0)$, denoting by $\mathcal{F}$ the Fourier transformation, we first write
	\begin{align*}
		E_k(0) &\geq \|w_k(x,D)\phi_k\|_{L^{2}(\R_x)} = \left\| h\left( \frac{x-4\sigma_{k}}{\sigma_{k}} \right) h\left( \frac{D - \sigma_k}{\frac{1}{4}\sigma_{k}} \right) \phi_k \right\|_{L^{2}(\R_x)} \\
		&= e^{-\rho_2 4^{\frac{1}{s}} \sigma^{\frac{1}{s}}_{k}} \left\| h\left( \frac{x-4\sigma_{k}}{\sigma_{k}} \right) h\left( \frac{D - \sigma_k}{\frac{1}{4}\sigma_{k}} \right) \phi(x-4\sigma_k) \right\|_{L^{2}(\R_x)} \\
		&= e^{-\rho_2 4^{\frac{1}{s}} \sigma^{\frac{1}{s}}_{k}} \left\| \mathcal{F} \left[ h \left(\frac{x-4\sigma_{k}}{\sigma_{k}} \right) \right] (\xi) \ast h\left( \frac{\xi-\sigma_k}{\frac{1}{4}\sigma_{k}} \right) e^{-i\xi 4\sigma_{k}}\widehat{\phi}(\xi)  \right\|_{L^{2}(\R_{\xi})} \\
		&= e^{-\rho_2 4^{\frac{1}{s}} \sigma^{\frac{1}{s}}_{k}} \sigma_{k} \left\| e^{-4i\sigma_{k} \xi}  \,\widehat{h}(\sigma_{k} \xi) \ast h\left( \frac{\xi-\sigma_k}{\frac{1}{4}\sigma_{k}} \right) e^{-4i\sigma_{k}\xi} \widehat{\phi}(\xi)  \right\|_{L^{2}(\R_{\xi})}.
	\end{align*}
	Thus
	\begin{equation*}
		E^{2}_k(0) \geq e^{-2\rho_2 4^{\frac{1}{s}}\sigma^{\frac{1}{s}}_{k}} \sigma^{2}_{k} \int_{\R_{\xi}} \left| \int_{\R_\eta} \widehat{h}(\sigma_{k}(\xi - \eta)) h\left( \frac{\eta-\sigma_k}{\frac{1}{4}\sigma_{k}} \right) \widehat{\phi}(\eta) d\eta
		\right|^{2} d\xi.
	\end{equation*}
	Since $\widehat{h}(0) > 0$ and $\widehat{h}(\xi) \geq 0$ for all $\xi \in \R$, we get an estimate from below to $E_k^2(0)$ by restricting the integration domain. Indeed, we set 
	$$
	G_{1,k} = \left[\sigma_k - \frac{\sigma_k}{8}, \sigma_k - \frac{\sigma_k}{8} + \sigma^{-2}_{k}\right] \bigcup \left[\sigma_k + \frac{\sigma_k}{8} - \sigma^{-2}_{k}, \sigma_k + \frac{\sigma_k}{8}\right],
	$$
	$$
	G_{2,k} = \left[\sigma_k - \frac{\sigma_k}{8}-\sigma^{-2}_{k}, \sigma_k - \frac{\sigma_k}{8} + \sigma^{-2}_{k}\right] \bigcup \left[\sigma_k + \frac{\sigma_k}{8} - \sigma^{-2}_{k}, \sigma_k + \frac{\sigma_k}{8}+\sigma^{-2}_{k}\right].
	$$
	Then $|\eta - \sigma_k| \leq 8^{-1}\sigma_k$ for all $\eta \in G_{1,k}$ and $\sigma_{k}|\xi-\eta| \leq 2\sigma^{-1}_{k}$ for all $\eta \in G_{1,k}$ and all $\xi \in G_{2,k}$. So, 
	\begin{align*}
		E^{2}_k(0) &\geq C e^{-2\rho_2 4^{\frac{1}{s}}\sigma^{\frac{1}{s}}_{k}} \sigma^{2}_{k} e^{-2c_{\rho_0} \sigma^{\frac{1}{\theta}}_{k}}\int_{G_{2,k}} \left| \int_{G_{1,k}}  d\eta \right|^{2} d\xi \\
		&= C \sigma^{-4}_{k} e^{-2\rho_2 4^{\frac{1}{s}}\sigma^{\frac{1}{s}}_{k}} e^{-2c_{\rho_{0}} \sigma_{k}^{\frac{1}{\theta}}}.
	\end{align*}
	In this way we get the following estimate
	\begin{align}\label{eq_almost_there_3}
		E_{k}(0) \geq C \sigma^{-2}_{k} e^{-\rho_2 4^{\frac{1}{s}}\sigma^{\frac{1}{s}}_{k}} e^{-c_{\rho_{0}} \sigma_{k}^{\frac{1}{\theta}}}.
	\end{align}
	
	From \eqref{eq_almost_there}, \eqref{eq_almost_there_2} and \eqref{eq_almost_there_3} we conclude 
	\begin{align}\label{estimate_from_below_E_k}
		E_k(T^*) \geq C e^{T^* \frac{c_1}{2} \sigma_{k}^{1-\sigma}} \left[ \sigma^{-2}_{k} e^{-\rho_2 4^{\frac{1}{s}}\sigma^{\frac{1}{s}}_{k}} e^{-c_{\rho_{0}} \sigma_{k}^{\frac{1}{\theta}}} - T^*e^{-c \sigma^{\frac{\lambda}{\theta_1}}_{k}} \right], \quad \text{for all}\,\, T^{*} \in (0,T],
	\end{align}
	provided that $\lambda < 1-\sigma$ and $k$ being sufficiently large.
	
	Now assume by contradiction that $\max\{\frac{1}{\theta},\frac{1}{s}\} < 1-\sigma$. Then we set $\lambda = \max\{\frac{1}{\theta},\frac{1}{s}\}+\tilde{\varepsilon}$ with $\tilde{\varepsilon} > 0$ small so that $\lambda < 1-\sigma$. After that, we take $\theta_1$ very close to $1$ to get $\frac{\lambda}{\theta_1} > \max\{\frac{1}{\theta},\frac{1}{s}\}$. Hence \eqref{estimate_from_below_E_k} implies
		\begin{align*}
			E_k(T^*) \geq C_2 e^{\tilde{c}_0 \sigma_{k}^{1-\sigma}} e^{-\tilde{c}_{\rho_0,\rho_2} \sigma^{\max\{\frac{1}{\theta},\frac{1}{s}\}}_{k} }.
		\end{align*}
		Thus
		\begin{align*}
			E_k(T^*) \geq C_3 e^{\frac{\tilde{c}_0}{2} \sigma_{k}^{1-\sigma}}.
		\end{align*}
		The latter inequality provides a contradiction because Proposition \ref{proposition_upper_bound_energy_E_k} provides a uniform upper bound for $E_k(t)$ for all $t$ and $k$.
	\end{proof} 
	
	\subsection{Proof of Proposition \ref{propostion_lower_bound_derivative_energy}}\label{Subsection}
	
	For sake of brevity we denote
	$$
	v^{(\alpha\beta)}_{k}(t,x) = w^{(\alpha\beta)}_{k}(x,D)u_k(t,x).
	$$
	Then we have 
	$$
	Sv^{(\alpha\beta)}_{k} = S w^{(\alpha\beta)}_{k} u_k = w^{(\alpha\beta)}_{k} \underbrace{S u_k}_{= 0} + [S, w^{(\alpha\beta)}_{k} ]u_k = [S, w^{(\alpha\beta)}_{k} ]u_k =: f^{(\alpha\beta)}_{k}.
	$$
	To obtain an estimate from below for $\partial_{t}E_{k}$ we observe
	\begin{align}\label{eq_energy_method}
		\|v^{(\alpha\beta)}_{k}\|_{L^{2}(\R)} \partial_t \|v^{(\alpha\beta)}_{k}\|_{L^{2}(\R)} &= \frac{1}{2}\partial_t \{ \| v^{(\alpha\beta)}_{k} \|^{2}_{L^{2}(\R)} \} 
		= \text{Re}\, \langle \partial_t v^{(\alpha\beta)}_{k}, v^{(\alpha\beta)}_{k} \rangle_{L^{2}(\R)} \nonumber \\
		&=\text{Re}\, \langle iSv^{(\alpha\beta)}_{k}, v^{(\alpha\beta)}_{k} \rangle_{L^{2}(\R)} + \text{Re}\, \langle \langle x \rangle^{-\sigma} D_{x} v^{(\alpha\beta)}_{k}, v^{\alpha\beta}_{k} \rangle_{L^{2}(\R)} \nonumber \\
		&\geq - \| f^{(\alpha\beta)}_{k}\|_{L^{2}(\R)} \| v^{(\alpha\beta)}_{k} \|_{L^{2}(\R)} + \text{Re}\, \langle \langle x \rangle^{-\sigma} D_{x} v^{(\alpha\beta)}_{k}, v^{\alpha\beta}_{k} \rangle_{L^{2}(\R)}.
	\end{align} 
	Next, our idea is to estimate $\text{Re}\, \langle \langle x \rangle^{-\sigma} D_{x} v^{(\alpha\beta)}_{k}, v^{(\alpha\beta)}_{k} \rangle_{L^2(\R)}$ from below and $\| f^{(\alpha\beta)}_{k}\|_{L^{2}(\R)}$ from above. In the following we shall derive such estimates. 
	
	From now and on we denote by $C$ a possibly large positive constant which do not depend on $\alpha,\beta, N_k$ and $k$. 
	
	\subsubsection{Estimate from below to $Re\, \langle \langle x \rangle^{-\sigma} D_{x} v^{(\alpha\beta)}_{k} , v^{(\alpha\beta)}_{k} \rangle_{L^{2}(\R)}$}
	We first consider the following cutoff functions 
	\begin{equation}\label{eq_cutoff_functions_to_split_the_supports}
		\chi_{k}(\xi) = h \left( \frac{\xi - \sigma_k}{ \frac{3}{4}\sigma_{k}} \right), \quad \psi_{k}(x) = h \left( \frac{x-4\sigma_{k}}{3\sigma_{k}} \right).
	\end{equation}
	On the support of $\psi_{k}(x) \chi_{k}(\xi)$ we have for all $k \in \N_0$
	$$
	\frac{\sigma_k}{4} \leq \xi \leq \frac{7\sigma_k}{4} \quad \text{and} \quad \sigma_{k} \leq x \leq 7 \sigma_{k}.
	$$ 
	Therefore 
	$$
	\xi \langle x \rangle^{-\sigma} \geq \frac{7^{-\sigma}}{4} \langle \sigma_{k} \rangle^{-\sigma} \sigma_k,
	$$
	for every $(x,\xi) \in \supp \psi_{k}(x) \chi_{k}(\xi)$.
	
	Denoting $c_0 = \frac{7^{-\sigma}}{4^{p-1}}$, we decompose the symbol of $\langle x \rangle^{-\sigma} D_{x}$ as follows 
	\begin{align*}
		\langle x \rangle^{-\sigma} \xi &= c_0\langle \sigma_{k} \rangle^{-\sigma} \sigma_{k} +  \langle x \rangle^{-\sigma} \xi - c_0\langle \sigma_{k} \rangle^{-\sigma} \sigma_{k} \\
		&= c_0\langle \sigma_{k} \rangle^{-\sigma} \sigma_{k} +  \{ \langle x \rangle^{-\sigma} \xi - c_0\langle \sigma_{k} \rangle^{-\sigma} \sigma_{k}\}\psi_{k}(x) \chi_{k}(\xi) \\
		&+ \{ \langle x \rangle^{-\sigma} \xi - c_0\langle \sigma_{k} \rangle^{-\sigma} \sigma_{k}\}\{1-\psi_{k}(x)\chi_{k}(\xi)\} \\
		&= I_{1,k} + I_{2,k}(x,\xi) + I_{3,k}(x,\xi).
	\end{align*}
	In the following we explain how to estimate $Re\, \langle I_{\ell,k}(x,D) v^{(\alpha\beta)}_{k}, v^{(\alpha\beta)}_{k} \rangle_{L^{2}(\R)}$, $\ell = 1, 2, 3$.
	\smallskip
	
	\noindent-$Re\, \langle I_{1,k} v^{(\alpha\beta)}_{k}, v^{(\alpha\beta)}_{k} \rangle_{L^{2}(\R)}$.
	For $I_{1,k}$ we simply have
	\begin{align}\label{eq_estimate_I_1_k}
		\text{Re}\, \langle I_{1,k} v^{(\alpha\beta)}_{k}, v^{(\alpha\beta)}_{k} \rangle_{L^2(\R)} &= c_0\langle \sigma_{k} \rangle^{-\sigma} \sigma_{k} \|v^{(\alpha\beta)}_{k}\|^{2}_{L^2(\R)} \nonumber \\
		&\geq c_0 2^{-\frac{\sigma}{2}} \sigma^{1-\sigma}_{k} \|v^{(\alpha\beta)}_{k}\|^{2}_{L^2(\R)}.
	\end{align}
	\smallskip
	
	\noindent-$Re\, \langle I_{2,k}(x,D) v^{(\alpha\beta)}_{k}, v^{(\alpha\beta)}_{k} \rangle_{L^{2}(\R)}$.
	We have that $I_{2,k}$ belongs to $S^{1}(\R^{2})$ with uniform symbol estimates with respect to $k$. Indeed, since $x$ and $\xi$ are comparable with $\sigma_{k}$ on the support of $\psi_k(x)\chi_k(\xi)$ we obtain 
	\begin{align*}
		|\partial^{\gamma}_{\xi}\partial^{\nu}_{x} I_{2,k}(x,\xi)| &\leq \sum_{\overset{\gamma_1+\gamma_2 = \gamma}{\nu_1+\nu_2 = \nu}} \frac{\gamma!\nu!}{\gamma_1!\gamma_2!\nu_1!\nu_2!} 
		|\partial^{\gamma_1}_{\xi} \partial^{\nu_1}_{x}\{ \langle x \rangle^{-\sigma} \xi - c_0\langle \sigma_{k} \rangle^{-\sigma} \sigma_{k}\}|
		|\partial^{\nu_2}_{x}\psi_{k}(x) \partial^{\gamma_2}_{\xi}\chi_{k}(\xi)| \\
		&\leq \sum_{\overset{\gamma_1+\gamma_2 = \gamma}{\nu_1+\nu_2 = \nu}} \frac{\gamma!\nu!}{\gamma_1!\gamma_2!\nu_1!\nu_2!} C^{\gamma_1+\nu_1+1} \gamma_1!\nu_1! \langle \xi \rangle^{1 - \gamma_1} \langle x \rangle^{-\sigma-\nu_1} C^{\gamma_2+\nu_2+1} \gamma_2!^{\theta_h} \nu_2!^{\theta_h} \sigma_k^{-\gamma_2} \sigma^{-\nu_2}_{k} \\
		&\leq C^{\gamma+\nu+1} \{\gamma!\nu!\}^{\theta_h} \langle \xi \rangle^{1-\gamma} \langle x \rangle^{-\sigma-\nu}.
	\end{align*} 
	Besides, from the choice of $c_0$ we easily conclude that $I_{2,k}(x.\xi) \geq 0$ for all $x, \xi \in \R$. Sharp-G{\aa}rding inequality then gives
	\begin{equation}\label{eq_estimate_I_2_k}
		Re\, \langle I_{2,k}(x,D) v^{(\alpha\beta)}_{k}, v^{(\alpha\beta)}_{k} \rangle_{L^{2}(\R)} \geq -C\|v^{(\alpha\beta)}_{k}\|^{2}_{L^{2}(\R)}.	
	\end{equation}	
	\smallskip
	
	\noindent-$Re\, \langle I_{3,k}(x,D) v^{(\alpha\beta)}_{k}, v^{(\alpha\beta)}_{k} \rangle_{L^{2}(\R)}$.
	Since the supports of $w^{(\alpha\beta)}_{k}(x,\xi)$ and of $1-\psi_{k}(x)\chi_{k}(\xi)$ are disjoint, applying formula \eqref{eq_asymptotic_expansion_formula} for $N := N_k$ (cf \eqref{eq_definition_of_N_k}) we may write
	\begin{align*}
		I_{3,k}(x,\xi) \circ w^{(\alpha\beta)}_{k}(x,\xi) = R^{(\alpha\beta)}_{k}(x,\xi), 
	\end{align*} 
	where 
	\begin{align*}
		R^{(\alpha\beta)}_{k}(x,\xi) = N_k\int_{0}^{1} \frac{(1-\theta)^{N_k-1}}{N_k!} \, Os - \iint e^{-iy\eta} \partial^{N_k}_{\xi} I_{3,k}(x,\xi+\theta\eta) D^{N_k}_{x}w^{(\alpha\beta)}_{k}(x+y,\xi)  dy\dslash\eta \, d\theta.
	\end{align*}
	The seminorms of $R_{k}$ can be estimated in the following way: for every $\ell_0 \in \N_0$ there exists $\ell_1 = \ell_1(\ell_0)$ such that  
	\begin{align*}
		|R_{k}|^{(0)}_{\ell_0} \leq C(\ell_0) \frac{N_k}{N_k!} |\partial^{N_k}_{\xi} I_{3,k}|^{(0)}_{\ell_1} |\partial^{N_k}_{x}w^{(\alpha\beta)}_{k}|^{(0)}_{\ell_1}.
	\end{align*}
	From Remark \ref{remark_about_the_w_k} we get
	$$
	|\partial^{N_k}_{x}w^{(\alpha\beta)}_{k}|^{(0)}_{\ell_1} \leq C^{\ell_1+\alpha+\beta+N_k+1} \ell_1!^{2\theta_h} \{\alpha!\beta!N_k!\}^{\theta_h} \sigma^{-N_k}_{k}.
	$$
	On the other hand, since there is no harm into assuming $N_k \geq 2$ (because $\sigma_k \to +\infty$), we have
	\begin{align*}
		\partial^{N_k}_{\xi} I_{3,k}(x,\xi) &= -\sum_{\overset{N_1+N_2=N_k}{N_1 \leq 1}} \frac{N_k!}{N_1!N_2!} \partial^{N_1}_{\xi}\{ \langle x \rangle^{-\sigma} \xi - c_0\langle \sigma_{k} \rangle^{-\sigma} \sigma_{k}\} \psi_{k}(x)\partial^{N_2}_{\xi}\chi_{k}(\xi),
	\end{align*}
	hence 
	$$
	|\partial^{N_k}_{\xi}I_3(x,\xi)|^{(0)}_{\ell_1} \leq C^{\ell_1+N_k+1}\ell_{1}!^{2\theta_h}N_k!^{\theta_h} \sigma^{1-N_k}_{k}.
	$$
	Therefore 
	$$
	|R^{(\alpha\beta)}_{k}|^{(0)}_{\ell_0} \leq C^{\alpha+\beta+N_k+1} \{\alpha!\beta!\}^{\theta_h} N_k!^{2\theta_h-1} \sigma^{1-2N_k}_{k},
	$$
	which allow us to conclude
	\begin{align*}
		\|I_{3,k}(x,D) v^{(\alpha\beta)}_{k}\|_{L^{2}(\R)} \leq C^{\alpha+\beta+N_k+1} \{\alpha!\beta!\}^{\theta_h} N_k!^{2\theta_h-1} \sigma^{1-2N_k}_{k} \| u_k\|_{L^{2}(\R)}.
	\end{align*}
	So, from \eqref{fat_girl} we get
	\begin{align}\label{eq_estimate_I_3_k}
		Re\, \langle I_{3,k}(x,D) v^{(\alpha\beta)}_{k}, v^{(\alpha\beta)}_{k} \rangle_{L^{2}(\R)} \geq 
		-C^{\alpha+\beta+N_k+1} \{\alpha!\beta!\}^{\theta_h} N_k!^{2\theta_h-1} \sigma^{1-2N_k}_{k} \|v^{(\alpha\beta)}_{k}\|_{L^{2}(\R)}.
	\end{align}
	
	From \eqref{eq_estimate_I_1_k}, \eqref{eq_estimate_I_2_k} and \eqref{eq_estimate_I_3_k} and using that $\sigma_k \to +\infty$ we obtain the following lemma.
	
	\begin{lemma}\label{lemma_estimate_that_gives_contradiction}
		Let the Cauchy problem \eqref{cauchy_problem_Model} be $\mathcal{S}^{\theta}_{s}$ well-posed. Then for all $k$ sufficiently large it holds 
		\begin{align*}
			Re\, \langle \langle x \rangle^{-\sigma} D_{x} v^{(\alpha\beta)}_{k} , v^{(\alpha\beta)}_{k} \rangle_{L^{2}(\R)} &\geq c_1 \sigma^{1-\sigma}_{k} \|v^{(\alpha\beta)}_{k}\|^{2}_{L^{2}(\R)} \\
			&-C^{\alpha+\beta+N_k+1} \{\alpha!\beta!\}^{\theta_h} N_k!^{2\theta_h-1} \sigma^{1-2N_k}_{k} \|v^{(\alpha\beta)}_{k}\|_{L^{2}(\R)},	
		\end{align*}
		for some $c_1 > 0$ independent from $k, \alpha, \beta$ and $N_k$.
	\end{lemma}
	
	\subsubsection{Estimate from above to $f^{(\alpha\beta)}_{k}$} We start recalling that
	$$
	f^{(\alpha\beta)}_{k} = [S, w^{(\alpha\beta)}_{k}]u_k = [D_t + D^{2}_x, w^{(\alpha\beta)}_{k}]u_k  + [i\langle x \rangle^{-\sigma}D_{x}, w^{(\alpha\beta)}_{k}]u_k.
	$$
	In the sequel we explain how to estimate the above brackets, but first we need the following lemma.
	
	\begin{lemma}\label{lemma_useful_estimate}
			Let the Cauchy problem \eqref{cauchy_problem_Model} be $\mathcal{S}^{\theta}_{s}$ well-posed. Then, for any $N = 1, 2, 3 \ldots$ we have
		\begin{equation*}
				\|D_{x} v^{(\alpha\beta)}_{k}\|_{L^{2}(\R)} \leq C \sigma_{k}\|v^{(\alpha\beta)}_{k}\|_{L^{2}(\R)} + C^{\alpha+\beta+N+1}\{\alpha!\beta!\}^{\theta_h} N!^{2\theta_h - 1} \sigma^{1-2N}_{k}.
		\end{equation*}
	\end{lemma}
	\begin{proof}
		We decompose $D_x$ as
		$$
		D_x = \chi_k(D)D_x + (1-\chi_k)(D) D_x.
		$$
		Then, since $\xi$ is comparable with $\sigma_k$ on the support of $\chi_k$, Calder\'on-Vaillancourt implies
		$$
		\|\chi_k(D)D_x v^{(\alpha\beta)}_{k}\|_{L^{2}(\R)} \leq C \sigma_k \|v^{(\alpha\beta)}_{k}\|_{L^{2}(\R)}.
		$$
		On the other hand, since the supports of $w^{(\alpha\beta)}_{k}$ and $1-\chi_k$ are disjoint, we may write 
		\begin{align*}
			(1-\chi_k)(\xi) \xi \circ w^{(\alpha\beta)}_{k}(x,\xi) = r^{(\alpha\beta)}_{k,N}(x,\xi),
		\end{align*}
		where 
		\begin{align*}
		|r^{(\alpha\beta)}_{k,N}|^{0}_{\ell_0} &\leq C \frac{N}{N!} |\partial^{N}_{\xi}\{(1-\chi_k(\xi))\xi\}|^{0}_{\ell_1} |\partial^{N}_{x}w^{(\alpha\beta)}_{k}|^{0}_{\ell_1} \\
		&\leq C^{\alpha+\beta+N+1} \{\alpha!\beta!\}^{\theta_h} N!^{2\theta_h-1} \sigma^{1-2N}_{k}.
		\end{align*}
	\end{proof}
	\smallskip
	
	\noindent-$[D_t + D^2_x, w^{(\alpha\beta)}_{k}]u_k$.
	We have
	\begin{align*}
		[D_t + D^2_x, w^{(\alpha\beta)}_{k}] &= 2 D_xw^{(\alpha\beta)}_{k} D_x + D^{2}_{x}w_{k}^{(\alpha\beta)} \\
		&= 2 D_x \circ D_x w^{(\alpha\beta)}_{k} - D^{2}_{x}w^{(\alpha\beta)}_{k} \\
		&= -2i \sigma^{-1}_{k} D_x \circ w^{((\alpha+1)\beta)}_{k} + \sigma^{-2}_{k} w^{((\alpha+2)\beta)}_{k}
	\end{align*}
	Thus, applying Lemma \ref{lemma_useful_estimate} with $N = N_k$,  we get
	\begin{align}\label{estimate_first_bracket_f_k_alpha_beta}
		\|[D_t+D^2_x, w^{(\alpha\beta)}_{k}]u_k\|_{L^{2}(\R)} &\leq C\|v^{((\alpha+1)\beta)}_{k}\|_{L^{2}(\R)} + C\sigma^{-2}_{k}\|v^{((\alpha+2)\beta)}_{k}\|_{L^{2}(\R)} \\
		&+ C^{\alpha+\beta+N_k+1} (\alpha!\beta!)^{\theta_h} N_k!^{2\theta_h-1} \sigma^{1-2N_k}_{k} \nonumber.
	\end{align}
	\smallskip
	
	\noindent-$[i\langle x \rangle^{-\sigma}D_{x}, w^{(\alpha\beta)}_{k}]u_k$.
	We first observe that
	\begin{align*}
		[i\langle x \rangle^{-\sigma}D_{x}, w^{(\alpha\beta)}_{k}] &= i\langle x \rangle^{-\sigma} D_xw^{(\alpha\beta)}_{k} 
		&- \sum_{1\leq \gamma \leq N_k-1} \frac{i}{\gamma!}  D^{\gamma}_{x}\langle x \rangle^{-\sigma} \partial^{\gamma}_{\xi}w^{(\alpha\beta)}_{k}(x,D) D_{x}  + r^{(\alpha\beta)}_{k}(x,D),
	\end{align*}
	where 
	$$
	r^{(\alpha\beta)}_{k}(x,\xi) = -iN_k\int_{0}^{1} \frac{(1-\theta)^{N_k-1}}{N_k!} \, Os - \iint e^{-iy\eta} \partial^{N_k}_{\xi}w^{(\alpha\beta)}_{k}(x,\xi+\theta\eta) D^{N_k}_{x} \langle x+y \rangle^{-\sigma} \xi dy\dslash\eta \, d\theta.
	$$
	To estimate $r^{(\alpha\beta)}_{k}$ we need to use the support properties of $w^{(\alpha\beta)}_{k}$. Then we write 
	$$
	\xi = (\xi + \theta\eta) - \theta\eta
	$$
	in order to use integration by parts to get
	\begin{align*}
		Os-\iint &e^{-iy\eta} \partial^{N_k}_{\xi}w^{(\alpha\beta)}_{k}(x,\xi+\theta\eta) D^{N_k}_{x} \langle x+y \rangle^{-\sigma} \xi dy\dslash\eta \\
		&= Os-\iint e^{-iy\eta} (\xi+\theta\eta)\partial^{N_k}_{\xi}w^{(\alpha\beta)}_{k}(x,\xi+\theta\eta) D^{N_k}_{x} \langle x+y \rangle^{-\sigma} \xi dy\dslash\eta \\
		&-\theta\, Os-\iint e^{-iy\eta} \partial^{N_k}_{\xi}w^{(\alpha\beta)}_{k}(x,\xi+\theta\eta) D^{N_k+1}_{x} \langle x+y \rangle^{-\sigma} \xi dy\dslash\eta .
	\end{align*}
	Hence we may estimate the seminorms of $r^{(\alpha\beta)}_{k}$ in the following way
	\begin{align*}
		|r^{(\alpha\beta)}_{k}|^{0}_{\ell_0} &\leq \frac{C^{N_k+1}}{N_k!} \left\{|\xi\partial^{N_k}_{\xi}w^{(\alpha\beta)}_{k}|^{(0)}_{\ell_1} |D^{N_k}_{x}\langle x \rangle^{-\sigma}|^{(0)}_{\ell_1} + |\partial^{N_k}_{\xi}w^{(\alpha\beta)}_{k}|^{(0)}_{\ell_1} |D^{N_k+1}_{x}\langle x \rangle^{-\sigma}|^{(0)}_{\ell_1} \right\} \\
		&\leq C^{\alpha+\beta+N_k+1} \{\alpha!\beta!\}^{\theta_h} N_k!^{2\theta_h-1} \sigma^{1-N_k}_{k},
	\end{align*}
	which allows us to conclude (using Calder\'on-Vaillancourt and \eqref{fat_girl})
	\begin{equation}\label{eq_estimate_main_remainder_f_k}
		\|r^{(\alpha\beta)}_{k,N}(x,D)u_{k}\|_{L^{2}} \leq C^{\alpha+\beta+N_k+1}\{\alpha!\beta!\}^{\theta_h} N_k!^{2\theta_h-1}\sigma^{1-N_k}_{k}.
	\end{equation}
	
	Now we consider the remaining terms of the commutator. We have 
	\begin{align*}
		\sum_{\gamma=1}^{N_k-1}& \frac{1}{\gamma!} D^{\gamma}_{x}\langle x \rangle^{-\sigma} \partial^{\gamma}_{\xi}w^{(\alpha\beta)}_{k}(x,D) D_{x} \\
		&= \sum_{\gamma = 1}^{N_k-1} \frac{1}{\gamma!} D^{\gamma}_{x}\langle x \rangle^{-\sigma} D_x \circ \partial^{\gamma}_{\xi} w^{(\alpha\beta)}_{k}(x,D)
		-  \sum_{\gamma = 1}^{N_k-1} \frac{1}{\gamma!} D^{\gamma}_{x}\langle x \rangle^{-\sigma} (D_x \partial^{\gamma}_{\xi}w)^{(\alpha\beta)}_{k}(x,D) \\
		&= \sum_{\gamma = 1}^{N_k-1} \frac{1}{\gamma!} 4^{\gamma} \sigma^{-\gamma}_{k} D^{\gamma}_{x}\langle x \rangle^{-\sigma} D_x \circ w^{(\alpha(\beta+\gamma))}_{k}(x,D) 
		+ i\sum_{\gamma = 1}^{N_k-1} \frac{1}{\gamma!} 4^{\gamma} \sigma^{-1-\gamma}_{k} D^{\gamma}_{x}\langle x \rangle^{-\sigma} w^{((\alpha+1)(\beta+\gamma))}_{k}(x,D)
	\end{align*}
	Using the support of $D_{x} v^{(\alpha(\beta+\gamma))}_{k}$ and that $|D^{\gamma}_{x} \langle x \rangle^{-\sigma}| \leq C^{\gamma+1} \gamma! \langle x \rangle^{-\sigma-\gamma}$ we get
	\begin{align*}
		\| D^{\gamma}_{x} \langle x \rangle^{-\sigma} D_{x} v^{(\alpha(\beta+\gamma))}_{k} \|_{L^{2}(\R)} \leq 
		C^{\gamma+1} \gamma! \sigma^{-(\sigma+\gamma)}_{k} \|D_{x}v^{(\alpha(\beta+\gamma))}_{k}\|_{L^{2}(\R)}.
	\end{align*}
	Then, applying Lemma \ref{lemma_useful_estimate} with $N = N_k - \gamma$ we obtain 
	\begin{align*}
		\| D^{\gamma}_{x} \langle x \rangle^{-\sigma} D_{x} &v^{(\alpha(\beta+\gamma))}_{k} \|_{L^{2}(\R)} \leq C^{\gamma+1} \gamma! \sigma^{1-(\sigma+\gamma)}_{k} \|v^{(\alpha(\beta+\gamma))}_{k}\|_{L^{2}(\R)} \\
		&+ C^{\gamma+1} \gamma! \sigma^{-(\sigma+\gamma)}_{k} C^{\alpha+\beta+N_k} \{\alpha!(\beta+\gamma)!\}^{\theta_h} (N_k-\gamma)!^{2\theta_h-1} \sigma^{1-2(N_k-\gamma)}_{k}. \\
	\end{align*}
	On the other hand, using the support of $v^{((\alpha+1)(\beta+\gamma))}_{k}$ we infer
	\begin{align*}
		\| D^{\gamma}_{x}\langle x \rangle^{-\sigma} v^{((\alpha+1)(\beta+\gamma))}_{k}\|_{L^{2}(\R)} \leq 
		C^{\gamma+1} \gamma! \sigma^{-(\sigma+\gamma)}_{k} \|v^{((\alpha+1)(\beta+\gamma))}_{k}\|_{L^{2}(\R)}.
	\end{align*}
	Hence
	\begin{align}\label{eq_estimate_shift_gamma_f_k}
		\Bigg\|\sum_{\gamma=1}^{N_k-1}& \frac{1}{\gamma!} D^{\gamma}_{x}\langle x \rangle^{-\sigma} \partial^{\gamma}_{\xi}w^{(\alpha\beta)}_{k}(x,D) D_{x} u_k\Bigg\|_{L^{2}(\R)} \leq 
		C\sum_{\gamma=1}^{N_k-1} C^{\gamma} \sigma^{1-\sigma-2\gamma}_{k} \|v^{(\alpha(\beta+\gamma))}_{k}\|_{L^{2}(\R)} \\ 
		&+ C\sum_{\gamma=1}^{N_k-1} C^{\gamma} \sigma^{-(1+\sigma+2\gamma)}_{k} \|v^{((\alpha+1)(\beta+\gamma))}_{k}\|_{L^{2}(\R)} \nonumber + C^{\alpha+\beta+N_k+1} \{\alpha!\beta!\}^{\theta_h} N_k!^{2\theta_h-1}\sigma_k^{1-\sigma-2N_k}.
	\end{align}
	
	For the last term we have
	\begin{align}\label{eq_last_estimate_f_k_alpha_beta}
			\|\langle x \rangle^{-\sigma}D_x v^{((\alpha+1)\beta)}_{k}\|_{L^{2}(\R)} \leq C \sigma^{-(1+\sigma)}_{k} \|v^{((\alpha+1)\beta)}_{k}\|_{L^{2}(\R)}.
	\end{align}
	
	From \eqref{estimate_first_bracket_f_k_alpha_beta}, \eqref{eq_estimate_main_remainder_f_k}, \eqref{eq_estimate_shift_gamma_f_k} and \eqref{eq_last_estimate_f_k_alpha_beta} we obtain the following lemma.
	
	\begin{lemma}\label{lemma_estimate_f_k_alpha_beta}
		Let the Cauchy problem \eqref{cauchy_problem_Model} be $\mathcal{S}^{\theta}_{s}$ well-posed.  We then have
		\begin{align*}
			\|f^{(\alpha\beta)}_{k}\|_{L^{2}(\R)} &\leq C\|v^{((\alpha+1)\beta)}_{k}\|_{L^{2}(\R)} + C\sigma^{-2}_{k}\|v^{((\alpha+2)\beta)}_{k}\|_{L^{2}(\R)} \\ \nonumber
			&+C\sum_{\gamma=1}^{N_k-1} C^{\gamma} \sigma^{1-\sigma-2\gamma}_{k} \|v^{(\alpha(\beta+\gamma))}_{k}\|_{L^{2}(\R)} + C\sum_{\gamma=1}^{N_k-1} C^{\gamma} \sigma^{-(1+\sigma+2\gamma)}_{k} \|v^{((\alpha+1)(\beta+\gamma))}_{k}\|_{L^{2}(\R)}\\ \nonumber
			&+ C^{\alpha+\beta+N_k+1}\{\alpha!\beta!\}^{\theta_h} N_k!^{2\theta_h-1}\sigma^{1-N_k}_{k}. \\
		\end{align*}	 
	\end{lemma}
	
	We are finally ready to prove Proposition \ref{propostion_lower_bound_derivative_energy}.
	
	\begin{proof}
	From \eqref{eq_energy_method} and Lemmas \ref{lemma_estimate_that_gives_contradiction} and \ref{lemma_estimate_f_k_alpha_beta} we have
	\begin{align*}
		\partial_t \|v^{(\alpha\beta)}_{k}\|_{L^{2}} &\geq c_1 \sigma^{1-\sigma}_{k} \|v^{(\alpha\beta)}_{k}\|_{L^{2}(\R)} 
		-C\|v^{((\alpha+1)\beta)}_{k}\|_{L^{2}(\R)} - C\sigma^{-2}_{k}\|v^{((\alpha+2)\beta)}_{k}\|_{L^{2}(\R)} \\
		&- C\sum_{\gamma=1}^{N_k-1} C^{\gamma} \sigma^{1-\sigma-2\gamma}_{k} \|v^{(\alpha(\beta+\gamma))}_{k}\|_{L^{2}(\R)} - C\sum_{\gamma=1}^{N_k-1} C^{\gamma} \sigma^{-(1+\sigma+2\gamma)}_{k} \|v^{((\alpha+1)(\beta+\gamma))}_{k}\|_{L^{2}(\R)} \\
		&- C^{\alpha+\beta+N_k+1}\{\alpha!\beta!\}^{\theta_h} N_k!^{2\theta_h-1}\sigma^{1-N_k}_{k}.
	\end{align*}
	Therefore
	\begin{align*}
		\partial_t &E_{k}(t) = \sum_{\alpha \leq N_k, \beta \leq N_k} \frac{1}{(\alpha!\beta!)^{\theta_1}} \partial_t \| v^{(\alpha\beta)}_{k} (t,x) \|_{L^{2}}  \\
		&\geq \sum_{\alpha \leq N+k, \beta \leq N_k} \frac{1}{(\alpha!\beta!)^{\theta_1}} c_1 \sigma^{1-\sigma}_{k} \|v^{(\alpha\beta)}_{k}\|_{L^{2}} \\
		&-C \sum_{\alpha \leq N_k, \beta \leq N_k} \frac{1}{(\alpha!\beta!)^{\theta_1}}\|v^{((\alpha+1)\beta)}_{k}\|_{L^{2}}
		-C \sigma^{-2}_{k} \sum_{\alpha \leq N_k, \beta \leq N_k} \frac{1}{(\alpha!\beta!)^{\theta_1}}  \|v^{((\alpha+2)\beta)}_{k}\|_{L^{2}} \\
		&-C \sum_{\alpha \leq N_k, \beta \leq N_k} \frac{1}{(\alpha!\beta!)^{\theta_1}} \left[\sum_{\gamma=1}^{N_k-1} C^{\gamma} \sigma^{1-\sigma-2\gamma}_{k} \|v^{(\alpha(\beta+\gamma))}_{k}\|_{L^{2}} + \sum_{\gamma=1}^{N_k-1} C^{\gamma} \sigma^{-(1+\sigma+2\gamma)}_{k} \|v^{((\alpha+1)(\beta+\gamma))}_{k}\|_{L^{2}} \right]\\
		&-C \sum_{\alpha \leq N_k, \beta \leq N_k} \frac{1}{(\alpha!\beta!)^{\theta_1}} C^{\alpha+\beta+N_k} (\alpha!\beta!)^{\theta_h} N_k!^{2\theta_h-1} \sigma^{1-N_k}_{k}.
	\end{align*}
	
	Now we discuss how to treat the terms appearing in the above summation. For the first one we simply have
	\begin{align}\label{eq_estimate_first_term}
		\sum_{\alpha \leq N_k, \beta \leq N_k} \frac{1}{(\alpha!\beta!)^{\theta_1}} c_1 \sigma^{1-\sigma}_{k} \|v^{(\alpha\beta)}_{k}\|_{L^{2}(\R)} = c_1 \sigma^{1-\sigma}_{k} E_k(t).
	\end{align}
	For the second one we proceed as follows
	\begin{align*}
		C \sum_{\alpha \leq N_k, \beta \leq N_k} \frac{1}{(\alpha!\beta!)^{\theta_1}}   \|v^{((\alpha+1)\beta)}_{k}\|_{L^{2}(\R)} &= 
		C \sum_{\alpha \leq N_k, \beta \leq N_k} (\alpha+1)^{\theta_1} E_{k,\alpha+1,\beta} \\
		&\leq C N^{\theta_1}_{k} \Bigg\{ E_{k} + \sum_{\beta \leq N_k} E_{k,N_k+1,\beta} \Bigg\}.
	\end{align*}
	From \eqref{eq_estimate_from_above_E_k_alpha_beta} we conclude 
	$$
	E_{k,\alpha+1,\beta} \leq C^{\alpha+\beta+1} \{(\alpha+1)!\beta!\}^{\theta_h-\theta_1},
	$$
	so we obtain
	\begin{align*}
		C \sum_{\alpha \leq N_k, \beta \leq N_k} \frac{1}{(\alpha!\beta!)^{\theta_1}}   \|v^{((\alpha+1)\beta)}_{k}\|_{L^{2}(\R)} &\leq  
		CN^{\theta_1} \Bigg\{E_{k} + C^{N_k} N_k!^{\theta_h-\theta_1} \sum_{\beta < \infty} C^{\beta} \beta!^{\theta_h-\theta_1} \Bigg\}.
	\end{align*}
	Recalling that $N_k := \lfloor \sigma_{k}^{\frac{\lambda}{\theta_1}} \rfloor$, the inequality $N_k^{N_k} \leq e^{N_k} N_k!$ implies
	\begin{align*}
		C \sum_{\alpha \leq N_k, \beta \leq N_k} \frac{1}{(\alpha!\beta!)^{\theta_1}}   \|v^{((\alpha+1)\beta)}_{k}\|_{L^{2}(\R)} &\leq 
		C \sigma^{\lambda}_{k}E_{k} + C^{N_k+1} \sigma_{k}^{\lambda + \frac{\lambda(\theta_h-\theta_1)}{\theta_1} N_k} .
	\end{align*}
	Hence 
	\begin{align}\label{eq_estimate_second_term}
		C \sum_{\alpha \leq N_k, \beta \leq N_k} \frac{1}{(\alpha!\beta!)^{\theta_1}}   \|v^{((\alpha+1)\beta)}_{k}\|_{L^{2}(\R)} &\leq 
		C \sigma^{\lambda}_{k} E_{k} + C^{N_k+1} \sigma^{C-cN_k}_{k}.
	\end{align}
	
	Analogously
	\begin{align}\label{eq_estimate_third_term}
		C \sigma^{-2}_{k} \sum_{\alpha \leq N_k, \beta \leq N_k} \frac{1}{(\alpha!\beta!)^{\theta_1}}  \|v^{((\alpha+2)\beta)}_{k}\|_{L^{2}} 
		\leq C \sigma^{2(\lambda-1)}_{k} E_k(t) + C^{N_k+1} \sigma^{C-cN_k}_{k}.
	\end{align}
	
	For the next term we first note
	\begin{align*}
		\sum_{\alpha \leq N_k, \beta \leq N_k} \frac{1}{(\alpha!\beta!)^{\theta_1}} \sum_{\gamma=1}^{N_k-1} &C^{\gamma} \sigma^{1-\sigma-2\gamma}_{k} \|v^{(\alpha(\beta+\gamma))}_{k}\|_{L^{2}}
		\\ 
		&= \sum_{\overset{\alpha,\beta \leq N_k}{1 \leq \gamma \leq N_k-1}} \left\{ \sum_{\beta+\gamma \leq N_k} + \sum_{\beta+\gamma > N_k} \right\}
		C^{\gamma} \sigma^{1-\sigma-2\gamma}_{k} \frac{(\beta+\gamma)!^{\theta_1}}{\beta!^{\theta_1}} E_{k,\alpha,\beta+\gamma}.
	\end{align*}	
	Now, since
	$$
	\frac{(\beta+\gamma)!}{\beta!} \leq (\beta+\gamma)^{\gamma} \leq N^{\gamma}_{k} \leq (\sigma^{\frac{\lambda}{\theta_1}}_{k})^{\gamma}, \quad \text{provided that}\, \beta+\gamma \leq N_k,
	$$
	$\lambda \in (0,1)$ and $\gamma \geq 1$, for $k$ large so that $C\sigma^{\lambda-1}_{k} < 1$ we get
	\begin{align*}
		\sum_{\overset{\alpha,\beta \leq N_k}{1 \leq \gamma \leq N_k-1}} \sum_{\beta+\gamma \leq N_k} C^{\gamma} \sigma^{1-\sigma-2\gamma}_{k} \frac{(\beta+\gamma)!^{\theta_1}}{\beta!^{\theta_1}} E_{k,\alpha,\beta+\gamma} \leq
		\sum_{\overset{\alpha,\beta \leq N_k}{1 \leq \gamma \leq N_k-1}} \sum_{\beta+\gamma \leq N_k} 
		(\sigma^{\lambda-1}_{k}C)^{\gamma} \sigma^{1-\sigma-\gamma}_{k} E_{k,\alpha,\beta+\gamma} \leq E_k.
	\end{align*}
	In the situation where $(\beta+\gamma) > N_k$ we use \eqref{eq_estimate_from_above_E_k_alpha_beta} to conclude
	\begin{align*}
		 C^{\gamma} \sigma^{1-\sigma-2\gamma}_{k} \frac{(\beta+\gamma)!^{\theta_1}}{\beta!^{\theta_1}} E_{k,\alpha,\beta+\gamma} &\leq 
		C^{\gamma} \sigma^{1-\sigma-2\gamma}_{k} \frac{(\beta+\gamma)!^{\theta_1}}{\beta!^{\theta_1}}
		C^{\alpha+\beta+\gamma+1} \alpha!^{\theta_h-\theta_1} (\beta+\gamma)!^{\theta_h-\theta_1}  \\
		&\leq C^{N_k+1} N_k!^{\theta_h-\theta_1} C^{\gamma} \sigma^{1-\sigma-2\gamma -\lambda\gamma}_{k} C^{\alpha}\alpha!^{\theta_h-\theta_1}.
	\end{align*}
	Hence we have
	\begin{align}\label{eq_estimate_fourth_term}
		\sum_{\alpha \leq N_k, \beta \leq N_k} \frac{1}{(\alpha!\beta!)^{\theta_1}} \sum_{\gamma=1}^{N_k-1} &C^{\gamma} \sigma^{1-\sigma-2\gamma}_{k} \|v^{(\alpha(\beta+\gamma))}_{k}\|_{L^{2}}
		\leq E_{k} + C^{N_k+1} \sigma^{C-cN_k}_{k}.
	\end{align}
	
	Analogously
	\begin{align}\label{eq_estimate_fiveth_term}
		\sum_{\alpha \leq N_k, \beta \leq N_k} \frac{1}{(\alpha!\beta!)^{\theta_1}} \sum_{\gamma=1}^{N_k-1} C^{\gamma} \sigma^{-(1+\sigma+2\gamma)}_{k} \|v^{((\alpha+1)(\beta+\gamma))}_{k}\|_{L^{2}}
		\leq E_{k} + C^{N_k+1} \sigma^{C-cN_k}_{k}.
	\end{align}
	
	For the last term, using the definition of $N_k$ and that $\theta_1 > \theta_h$ we easily conclude 
	\begin{align}\label{eq_estimate_sixth_term}
		\sum_{\alpha \leq N_k, \beta \leq N_k} \frac{1}{(\alpha!\beta!)^{\theta_1}} C^{\alpha+\beta+N_k+1} (\alpha!\beta!)^{\theta_h} N_k!^{2\theta_h-1} \sigma^{1-N_k}_{k} 
		\leq C^{N_k+1} \sigma^{C-cN_k}_{k}.
	\end{align}
	
	Finally, associating \eqref{eq_estimate_first_term},\eqref{eq_estimate_second_term},\eqref{eq_estimate_third_term}, \eqref{eq_estimate_fourth_term}, \eqref{eq_estimate_fiveth_term} and \eqref{eq_estimate_sixth_term} and using that $\sigma_k \to \infty$ we close the proof of Proposition \ref{propostion_lower_bound_derivative_energy}.
	
	\end{proof}

	\section{Proof of Proposition \ref{main_prop_2}}\label{section_proof_prop_2}
	
	\begin{proof}
		For a fixed $t > 0$ consider the function 
		$$
		h(\xi) = e^{-it\xi^2}, \quad \xi \in \R.
		$$
		We shall prove that $h(\xi)$ does not define a multiplier in the space $\mathcal{S}^{s}_{\theta}(\R)$ when $1\leq s < \theta$.  
		%We always shall consider the parameters $s$ and $\theta$ satisfying $1 \leq s < \theta$.
		
		Notice that we have the following formula for the derivatives of $h$ (cf. \cite{wahlberg2022semigroups} Eq. $6.3$):
		\begin{equation}\label{eq_expression_derivatives_of_g}
			\partial^{\alpha}_{\xi} h(\xi) = h(\xi) \, \underbrace{(-2it\xi)^{\alpha}\, \sum_{m = 0}^{\lfloor \frac{\alpha}{2} \rfloor} \frac{1}{(-4it)^{m}} \, \frac{\alpha!}{m!(\alpha-2m)!} \, \xi^{-2m}}_{=:P_{\alpha}(\xi)}.
		\end{equation}
		
		Next we consider the function 
		$$
		f(\xi) = e^{-\langle \xi \rangle^{\frac{1}{\theta}}}, \quad \xi \in \R.
		$$
		Of course we have $f \in \mathcal{S}^{1}_{\theta}(\R) (\subset \mathcal{S}^{s}_{\theta}(\R))$ and Fa\`a di Bruno formula implies
		$$
		|\partial^{\alpha}_{\xi} f(\xi)| \leq B^{\alpha} \alpha! f(\xi), \quad \xi \in \R, \alpha \in \N_0,
		$$
		for some $B > 0$.
		
		Suppose by contradiction that $h$ defines a multiplier in $\mathcal{S}^{s}_{\theta}(\R)$. Then, enlarging $B > 0$ if necessary,
		\begin{equation}\label{eq_estimate_assuming_g_multiplier}
			|\partial^{\alpha}_{\xi} \{h(\xi) f(\xi)\}| \leq A B^{\alpha} \alpha!^{s} e^{-a\langle \xi \rangle^{\frac{1}{\theta}}}, \quad \xi \in \R,\, \alpha \in \N_0,
		\end{equation}
		for some $A, a > 0$. Now, we shall prove by induction on $\alpha \in \N_0$ that
		\begin{equation}\label{eq_auxilliary_estimate}
			|\{\partial^{\alpha}_{\xi}h(\xi)\} f(\xi)| \leq A (2B)^{\alpha} \alpha!^{s} e^{-a\langle \xi \rangle^{\frac{1}{\theta}}}, \quad \xi \in \R,\, \alpha \in \N_0.
		\end{equation}
		The case $\alpha = 0$ follows from \eqref{eq_estimate_assuming_g_multiplier}. To prove for a general $\alpha \geq 1$ we first write
		$$
		\{\partial^{\alpha}_{\xi} h(\xi)\} f(\xi) = \partial^{\alpha}_{\xi}\{h(\xi) f(\xi)\} - \sum_{ \overset{\alpha' +\alpha'' = \alpha}{\alpha'' \geq 1}} \frac{\alpha!}{\alpha'!\alpha''!}\partial^{\alpha'}_{\xi}h(\xi) \partial^{\alpha''}_{\xi} f(\xi)
		$$
		and then we use the induction hypoyhesis to get ($\alpha' < \alpha$)
		\begin{align*}
			|\{\partial^{\alpha}_{\xi} h(\xi)\} f(\xi)| &\leq AB^{\alpha}\alpha!^{s}e^{-a\langle \xi \rangle^{\frac{1}{\theta}}} 
			+ \sum_{ \overset{\alpha' +\alpha'' = \alpha}{\alpha'' \geq 1}} \frac{\alpha!}{\alpha'!\alpha''!} |\{\partial^{\alpha'}_{\xi}h(\xi)\}f(\xi)| B^{\alpha''} \alpha''! \\
			&\leq AB^{\alpha}\alpha!^{s}e^{-a\langle \xi \rangle^{\frac{1}{\theta}}} 
			+ \sum_{ \overset{\alpha' +\alpha'' = \alpha}{\alpha'' \geq 1}} \frac{\alpha!}{\alpha'!\alpha''!} A (2B)^{\alpha'} \alpha'!^{s} e^{-a\langle \xi \rangle^{\frac{1}{\theta}}} B^{\alpha''} \alpha''! \\
			&= AB^{\alpha}\alpha!^{s}e^{-a\langle \xi \rangle^{\frac{1}{\theta}}} \left( 1 +  \sum_{ \overset{\alpha' +\alpha'' = \alpha}{\alpha'' \geq 1}} \frac{\alpha!}{\alpha'!\alpha''!} \frac{2^{\alpha'}\alpha'!^{s}\alpha''!}{\alpha!^{s}} \right).
		\end{align*}
		Since $s \geq 1$ we obtain
		$$
		1 +  \sum_{ \overset{\alpha' +\alpha'' = \alpha}{\alpha'' \geq 1}} \frac{\alpha!}{\alpha'!\alpha''!} \frac{2^{\alpha'}\alpha'!^{s}\alpha''!}{\alpha!^{s}} 
		\leq 1 + \sum_{ \overset{\alpha' +\alpha'' = \alpha}{\alpha'' \geq 1}}  2^{\alpha'}
		= 1 + \sum_{\alpha' = 0}^{\alpha-1} 2^{\alpha'} = 1 + 2^{\alpha} - 1 = 2^{\alpha},
		$$
		which gives \eqref{eq_auxilliary_estimate}.
		
		In the sequel our idea is to prove that  the inequality \eqref{eq_auxilliary_estimate} does not hold for the sequence $\xi_{\alpha} = t^{-\frac{1}{2}}\alpha^{\theta}$, $\alpha \in \N_0$, provided that $\alpha$ is large enough. First we use \eqref{eq_expression_derivatives_of_g} to write 
		\begin{align*}
		|\{\partial^{\alpha}_{\xi}h\}(t^{-\frac{1}{2}}\alpha^\theta)f(t^{-\frac{1}{2}}\alpha^{\theta})| &= |h(t^{-\frac{1}{2}}\alpha^\theta) P_{\alpha}(t^{-\frac{1}{2}}\alpha^{\theta}) f(t^{-\frac{1}{2}}\alpha^{\theta})| = |P_{\alpha}(t^{-\frac{1}{2}}\alpha^\theta)| e^{-\langle t^{-\frac{1}{2}}\alpha^{\theta} \rangle^{\frac{1}{\theta}}}. 
		%\leq \{e^{-c_t}\}^{\alpha} |P(t^{-\frac{1}{2}}\alpha^{\theta})|.
		\end{align*}
		Since we are interested in $\alpha$ large, there is no harm in assume $t^{-\frac{1}{2}}\alpha^{\theta} > 1$, hence 
		$$
		|\{\partial^{\alpha}_{\xi}h\}(t^{-\frac{1}{2}}\alpha^\theta)f(t^{-\frac{1}{2}}\alpha^{\theta})| \geq \{e^{- (2t^{-1})^{\frac{1}{2\theta}}}\}^{\alpha}|P_{\alpha}(t^{-\frac{1}{2}}\alpha^{\theta})|.
		$$
		The next step is to obtain an estimate from below for $|P_{\alpha}(t^{-\frac{1}{2}}\alpha^{\theta})|$. Let us then study the expression of $P_{\alpha}(t^{-\frac{1}{2}}\alpha^\theta)$:
		\begin{align*}
			P_{\alpha}(t^{-\frac{1}{2}}\alpha^\theta) &= (-2it^{\frac{1}{2}}\alpha^\theta)^{\alpha}\, \sum_{m = 0}^{\lfloor \frac{\alpha}{2} \rfloor} \frac{1}{(-4i)^{m}} \, \frac{\alpha!}{m!(\alpha-2m)!} \, \alpha^{-2\theta m} \\
			&= (-2it^{\frac{1}{2}}\alpha^\theta)^{\alpha} \left\{ \sum_{\overset{m = 0}{m \, \text{even}}}^{\lfloor \frac{\alpha}{2} \rfloor} \underbrace{\frac{1}{(4i)^{m}} \, \frac{\alpha!}{m!(\alpha-2m)!} \, \alpha^{-2\theta m}}_{\in \R}
			+ i \sum_{\overset{m = 0}{m \, \text{odd}}}^{\lfloor \frac{\alpha}{2} \rfloor} \underbrace{\frac{1}{-4^{m}i^{m-1}} \, \frac{\alpha!}{m!(\alpha-2m)!} \, \alpha^{-2\theta m}}_{\in \R} \right\}.
		\end{align*}
		Therefore
		\begin{align*}
			|P_{\alpha}(t^{-\frac{1}{2}}\alpha^\theta)| \geq (2t^{\frac{1}{2}})^{\alpha} \alpha^{\theta\alpha} \left| \sum_{\overset{m = 0}{m \, \text{even}}}^{\lfloor \frac{\alpha}{2} \rfloor} \frac{1}{(4i)^{m}} \, \frac{\alpha!}{m!(\alpha-2m)!} \, \alpha^{-2\theta m} \right|.
		\end{align*}
		Next we claim that 
		\begin{equation}\label{estimate_1}
			\left| \sum_{\overset{m = 0}{m \, \text{even}}}^{\lfloor \frac{\alpha}{2} \rfloor} \frac{1}{(4i)^{m}} \, \frac{\alpha!}{m!(\alpha-2m)!} \, \alpha^{-2\theta m} \right| \geq \frac{3}{4}.
		\end{equation}
		Indeed, first observe that the sequence 
		$$
		a_{m,\alpha} := \frac{\alpha!}{m!(\alpha-2m)!}, \quad m = 0, 1, \ldots, \lfloor \alpha/2 \rfloor,
		$$
		is strictly decreasing: for any $0 \leq m \leq \lfloor \alpha/2 \rfloor - 1$ we have ($\theta > 1$)
		\begin{align*}
			a_{m,\alpha} > a_{m+1, \alpha} &\iff \frac{\alpha!}{m!(\alpha-2m)!} \, \alpha^{-2\theta m} > \frac{\alpha!}{(m+1)!(\alpha-2(m+1))!} \, \alpha^{-2\theta(m+1)} \\
			&\iff \underbrace{(m+1)}_{\geq 1} \underbrace{\frac{\alpha^\theta}{\alpha-2m} \frac{\alpha^\theta}{\alpha-2m-1}}_{>1} > 1.
		\end{align*}
		We also have $a_{0,\alpha} = 1$ and $a_{m,\alpha} \in (0,1]$. Thus we conclude 
		$$
		\frac{1}{4^{m}}a_{m,\alpha} - \frac{1}{4^{m+2}}a_{m+2,\alpha} \geq \left(\frac{1}{4^{m}} - \frac{1}{4^{m+2}}\right) a_{m,\alpha} \geq 0.
		$$
		Hence
		$$
		\sum_{\overset{m = 0}{m \, \text{even}}}^{\lfloor \frac{\alpha}{2} \rfloor} \frac{1}{(4i)^{m}} \, \frac{\alpha!}{m!(\alpha-2m)!} \, \alpha^{-2\theta m} 
		= \underbrace{1 - \frac{1}{4^{2}} a_{2,\alpha}}_{\geq \frac{3}{4}} + \underbrace{\frac{1}{4^{4}} a_{4,\alpha} - \frac{1}{4^{6}} a_{6,\alpha} + \cdots}_{\geq 0}.
		$$
		Finally, from \eqref{estimate_1} we get
		\begin{equation}\label{eq_estimate_from_below_P_alpha}
			|P_{\alpha}(t^{-\frac{1}{2}}\alpha^\theta)| \geq \frac{3}{4} \{2t^{\frac{1}{2}}e^{- (2t^{-1})^{\frac{1}{2\theta}}}\}^{\alpha} \alpha^{\theta\alpha}.
		\end{equation}
		
		Associating \eqref{eq_auxilliary_estimate} with \eqref{eq_estimate_from_below_P_alpha} we obtain for all $\alpha$ large
		$$
		\frac{3}{4} \{2t^{\frac{1}{2}}e^{- (2t^{-1})^{\frac{1}{2\theta}}}\}^{\alpha} \alpha^{\theta\alpha} \leq A (2B)^{\alpha} \alpha!^{s} e^{-a \langle t^{-\frac{1}{2}} \alpha^{\theta} \rangle^{\frac{1}{\theta}}},
		$$
		which is a contradiction because we are assuming $1 \leq s < \theta$. In this way $h(\xi)f(\xi)$ cannot belong to $\mathcal{S}^{s}_{\theta}(\R)$. In particular, $h$ does not define a multiplier in $\mathcal{S}^{s}_{\theta}(\R)$, when $1 \leq s < \theta$.
		
	\end{proof}


\begin{thebibliography}{AAA}
		\bibitem{AAC_necessary}
		A. Arias Junior, A.Ascanelli, M. Cappiello, {\it On the Cauchy problem for $p$-evolution equations with variable coefficients: a necessary condition for Gevrey well-posedness}, 2023. arxiv: \url{https://arxiv.org/abs/2309.05571}
		
		\bibitem{ascanelli_cappiello_schrodinger_equations_Gelfand_shillov}
		A. Ascanelli, M. Cappiello, {\em Schr{\"o}dinger-type equations in Gelfand-Shilov spaces.} 
		 J. Math. Pures Appl. \textbf{132} (2019), 207-250.
		
		\bibitem{dreher}
		M. Dreher, {\em Necessary conditions for the well-posedness of Schr{\"o}dinger type equations in Gevrey spaces.}
		Bulletin des sciences mathematiques \textbf{127} (6) 2003, 485-503.
		
		\bibitem{GS2} I.M.~Gelfand, G.E.~Shilov, {\it Generalized functions, Vol. 2}, Academic
		Press, New York-London, 1967.
		
		\bibitem{ichinose_remarks_cauchy_problem_schrodinger_necessary_condition}
		W. Ichinose, {\em Some remarks on the Cauchy problem for Schr{\"o}dinger type equations.}
		Osaka J.  Math. \textbf{21} (3) (1984) 565-581.
		
		\bibitem{Ichinose_sufficient}
		W. Ichinose, {\em Sufficient condition on $H_{\infty}$ well posedness for Schr\"odinger type equations.}
		Comm. Partial Differential Equations 9 (1) (1984) 33-48.
		
		\bibitem{KB}
		K. Kajitani, A. Baba, {\em The Cauchy problem for Schr{\"o}dinger type equations.}
		Bull. Sci. Math.\textbf{119} (5), (1995), 459-473.
		
		\bibitem{Kumano-Go}
		H. Kumano-Go. \textit{Pseudo-differential operators}. The MIT Press, Cambridge, London, 1982.
		
		\bibitem{mizohata2014}
		S. Mizohata, {\it On the Cauchy problem}, Academic Press. Volume 3 (2014).
		
		\bibitem{nicola_rodino_global_pseudo_diffferential_calculus_on_euclidean_spaces}
		F. Nicola, L. Rodino, {\em Global pseudo-differential calculus on Euclidean spaces.}
		Series Pseudo-Differential Operators \textbf{4},  Springer, 2010.
		
		\bibitem{wahlberg2022semigroups}
		P. Wahlberg, {\em Semigroups for quadratic evolution equations acting on Shubin-Sobolev and Gelfand-Shilov spaces}.
		Annales Fennici Mathematici, 47 (2022) 821-853.
	\end{thebibliography}
\end{document}